\documentclass[article]{IEEEtran}

\IEEEoverridecommandlockouts
\usepackage[noadjust]{cite}
\usepackage{amsmath,amssymb,amsfonts}
\usepackage{algorithmic}
\usepackage{graphicx}
\usepackage{textcomp}
\usepackage{xcolor}
\usepackage{dsfont}
\usepackage{amsthm}
\usepackage{enumitem}
\usepackage{subcaption}
\usepackage{mathbbol,eufrak} 

\usepackage[font=small]{caption}
\usepackage{tikz}
\usepackage{onimage}

\newtheorem{theorem}{Theorem}[section]
\newtheorem{corollary}{Corollary}[theorem]
\newtheorem{lemma}[theorem]{Lemma}
\newtheorem{remark}{Remark}[section]
\newtheorem{proposition}[theorem]{Proposition}
\newtheorem{assumption}{Assumption}

\newtheorem{example}{Example}[section]

\newcommand{\bs}[1]{\boldsymbol{#1}}
\newcommand{\R}{{\mathds{R}}}

\newcommand{\D}{{\mathfrak D}}
\newcommand{\Z}{{\mathfrak Z}}
\newcommand{\Ga}{{\mathfrak G}}

\newcommand{\ES}{{\mathfrak S}}
\newcommand{\VV}{{\mathcal V}}

\newcommand{\Na}{{N_{\rm a}}}
\newcommand{\No}{{N_{\rm o}}}

\newcommand{\Zz}{\bs{Z}}

\definecolor{mygreen}{rgb}{0.1, 0.45, 0.1}

\newcommand{\newtext}[1]{{\color{black}{#1}}}

\begin{document}

\title{Nonlinear Opinion Dynamics with Tunable Sensitivity\thanks{This research has been supported in part by NSF grant CMMI-1635056, ONR grant N00014-19-1-2556, ARO grant W911NF-18-1-0325, DGAPA-UNAM PAPIIT grant IN102420, Conacyt grant A1-S-10610,and by NSF Graduate Research Fellowship DGE-2039656. Any opinion, findings, and conclusions or recommendations expressed in this material are those of the authors and do not necessarily reflect the views of the NSF.}}

\author{Anastasia~Bizyaeva,
        Alessio~Franci,
        and~Naomi~Ehrich~Leonard
\thanks{A. Bizyaeva and N.E. Leonard are with the Department
of Mechanical and Aerospace Engineering, Princeton University, Princeton,
NJ, 08544 USA; e-mail: bizyaeva@princeton.edu, naomi@princeton.edu.}
\thanks{A. Franci is with Math Department of the National Autonomous University of Mexico, 04510 Mexico City, Mexico. e-mail: afranci@ciencias.unam.mx}

}

\maketitle

\begin{abstract}
\newtext{We propose a continuous-time multi-option nonlinear generalization of classical linear weighted-average opinion dynamics. Nonlinearity is introduced by saturating opinion exchanges, and this is enough to enable a significantly greater range of opinion-forming behaviors with our model as compared to existing linear and nonlinear models. For a group of agents that communicate opinions over a network, these behaviors include multistable agreement and disagreement, tunable sensitivity to input, robustness to disturbance, flexible transition between patterns of opinions, and opinion cascades.  We derive network-dependent  tuning rules to robustly control the system behavior and we design state-feedback dynamics for the model parameters to make the behavior adaptive to changing external conditions.} The  model provides new means for systematic study of dynamics on natural and engineered networks, from information spread and political polarization to collective decision making and dynamic task allocation.
\end{abstract}

\section{Introduction.}

Opinion dynamics of networked agents are the subject of long-standing  interdisciplinary interest, and there is a large and growing literature on agent-based models created to study mechanisms that drive the formation of consensus  and opinion clustering in groups. These models appear, for example, in studies of collective animal behavior and voting patterns in human social networks. In engineering, they are fundamental to designing distributed coordination of autonomous agents and dynamic allocation of tasks across a network.

Agent-based models are typically used to investigate parameter regimes and network structures for which opinions in a group converge over time to a desired configuration. However, natural groups exhibit much more flexibility than captured with existing models. Remarkably, groups in nature can rapidly switch between different opinion configurations in response to  changes in their environment, and they can break deadlock, i.e., choose among options with little, if any, evidence that one option is better than another. Understanding the mechanisms that explain the temporal dynamics of opinion formation in groups and  the ultra-sensitivity and robustness needed for groups to pick out meaningful information and to break deadlock in uncertain and changing environments is important in its own right. It is also pivotal to developing  the means to design provably adaptable yet robust control laws for robotic teams and other networked multi-agent systems. 

Motivated by these observations, we explore the following questions in this paper. How can a network of decision makers come rapidly and reliably to  coherent configurations of opinions, including both agreement and disagreement, 
on multiple options in response to, or in the absence of, internal biases or external inputs? How can a network reliably transition from one configuration of opinions to another in response to   change? How can the sensitivity of the opinion formation process be tuned so that meaningful signals are distinguished from spurious signals?  To investigate these questions, we present an agent-based dynamic model of the opinion formation process that generalizes linear and existing nonlinear models. 
The model is rich in the behaviors it exhibits yet tractable to analysis by virtue of the small number of parameters needed to generate the full range of behaviors. 

\newtext{We emphasize that our modeling approach is distinct from existing models in the literature in the following way. Models of opinion formation are typically built on the fundamental assumption that individuals update their opinions through a linear averaging process ~\cite{DeGroot1974, FriedkinJohnsen1999,CisnerosVelarde2019PolarizationAF,OlfatiSaber2004,Altafini2013}. Additional feedback dynamics are then often imposed on the coupling weights between agents,  for example in bounded confidence models \cite{Deffuant2000,HegselmannKrause2002,hegselmann_opinion_2005, BlondelTAC2009}, biased assimilation models \cite{Dandekar2013,xia2020analysis}, and models of evolution of social power \cite{Peng2015,Basar2018}. Nonlinearity thereby arises through the superposition of linear opinion dynamics and nonlinear coupling-weight dynamics. When persistent disagreement is observed, it is necessarily the consequence of the dynamic updating of the coupling weights. However, state-dependent interactions are not the only way for a network to achieve structurally stable disagreement. We are instead proposing that the opinion update process itself is fundamentally nonlinear due to saturation of information. We introduce a new multi-option nonlinear model of opinion formation with saturated interactions in Section \ref{sec: multi option opinion dynamics} and in Section \ref{sec:agree_disagree} we prove that this modeling assumption supports persistent disagreement with a completely static interaction network. 

As is done for linear models, dynamic feedback can also be introduced to the nonlinear model parameters. 
We  explore the effects of several dynamic parameter update laws in detail in Section \ref{sec:feedback}. The feedback laws we consider are simple, yet they make our model adaptive to changing external conditions with tunable sensitivity and they allow robust and tunable transitions between distinctly different patterns of opinions.


Our model generalizes  recent literature on opinion formation with input saturation \cite{AF-VS-NEL:15a,fontan2017multiequilibria,Gray2018,Abara2018,Fontan2018,ding2018network,hu2018event}. Closely related to these are nonlinear models that leverage coupled oscillator dynamics \cite{MirolloStrogatz1990,Nabet2009,Leonard2012}, biologically inspired mean-field models \cite{Pais2013}, and the Ising model 
\cite{Gov2018,Siegenfeld2020}.} 


Our major contributions are as follows.
1) We introduce a new \newtext{nonlinear} model for the study of multi-agent, multi-option opinion dynamics. 
The model has
a social term weighted by an attention parameter, which can also represent social effort or susceptibility to social influence, and an input term, which can represent, e.g., external stimuli, bias, or persistent opinions.  

2) We show that the model exhibits a rich variety of opinion-formation behaviors governed by bifurcations. This includes rapid and reliable opinion formation and  multistable agreement and disagreement, with flexible transitions between them. It also includes ultra-sensitivity to inputs near the opinion forming bifurcation, and robustness
to  disturbances and uncertainties, away from the bifurcation. Moreover, the behaviors are governed by a small number of key parameters, rendering the model analytically tractable. We prove the central role of the spectral properties of the network graph adjacency matrix in informing the model behavior. 

3) We show how the model recovers a range of models in the literature for suitable parameter combinations and/or when linearized, 
and how the reliance on structurally unstable network conditions
in linear models
breaks down in the nonlinear setting.  
\newtext{The central role of the network graph adjacency matrix in our nonlinear model generalizes the central role of the network graph Laplacian in opinion dynamics in the literature. We show that the right and left adjacency matrix eigenstructures determine patterns of opinion and sensitivity to inputs, respectively.}

4) We introduce distributed adaptive feedback dynamics to the agent 
parameters. 
We show how design parameters in the attention feedback  allow tunable sensitivity of opinion formation to inputs and robustness to changes in inputs,  as well as tunable opinion cascades even in response to a single agent receiving an input.  

5) 
We examine tunable transitions between consensus and dissensus using feedback dynamics also on network weights.

\newtext{
We define notation in Section~\ref{sec:notation}. We present the new nonlinear opinion dynamics model in Section~\ref{sec: multi option opinion dynamics}. In Section~\ref{sec:agree_disagree} we prove results on agreement and disagreement opinion formation for the new model. We introduce attention dynamics and prove results on tunable sensitivity in Section~\ref{sec:feedback} in the special case of two options.
In Section~\ref{sec:transition}, we  illustrate feedback controlled transitions between agreement and disagreement. We conclude in Section~\ref{sec: final remarks}.}
\newtext{
\section{Notation}
\label{sec:notation}

Given  $\mathbf{y} \in \mathds{R}^{n}$, the norm $\| \mathbf{y} \|$ is the standard Euclidean 2-norm and $\operatorname{diag}\{\mathbf{y} \} \in \mathds{R}^{n \times n}$ is a diagonal matrix with $y_i$ in row $i$, column $i$. Let $\mathcal{I}_N \in \R^{N \times N}$ be the identity matrix, $\mathbf{1}_N \in \R^N$  the vector of ones, and $P_0 = (\mathcal{I}_{\No} - \frac{1}{\No}\mathbf{1}_{\No}\mathbf{1}_{\No}^T)$  the projection onto $\mathbf{1}_\No^{\perp}$. Let $\R\{\mathbf{v}_1,\ldots,\mathbf{v}_k\}$ be the span of  vectors $\mathbf{v}_1,\ldots,\mathbf{v}_k\in\R^n$. We define  $\mathbf{v}_i \in \mathds{R}^n$ component-wise as $(v_{i1}, \dots, v_{i n})$. 
Let $U$, $V$ and $W$ be vector spaces. $U$ is the \textit{direct sum} of $V$ and $W$, i.e., $U = V \oplus W$, if and only if $V = U + W$ and $U \cap W = \{ 0 \}$. Given matrices $B = (b_{ij}) \in \mathds{R}^{m \times n}$ and $C = (c_{ij}) \in \mathds{R}^{p \times q}$, the \textit{Kronecker product} $B \otimes C \in \mathds{R}^{mp\times nq}$ has entries $(B \otimes C)_{pr+v,qs+w} = b_{rs} c_{vw}$. 

Let the set of vertices $V = \{ 1, \dots, N_{a} \}$ index a group of $N_{a}$ agents, and let edges $E \subseteq V \times V$ represent interactions between agents. If edge $e_{ik} \in  E$, then agent $k$ is a \textit{neighbor} of agent $i$. The communication topology between agents is captured by the \textit{directed} graph $ G = (V,  E)$ and its associated adjacency matrix $A \in \mathds{R}^{N_a\times N_a}$. $A$ is made up of elements $ a_{ik}$, and $a_{ik} \neq 0$ if and only if agent $k$ is a neighbor of agent $i$. When  $A$ is symmetric (i.e., communication between agents is bidirectional), the graph is \textit{undirected}. 
}

\section{Nonlinear multi-option opinion dynamics}
\label{sec: multi option opinion dynamics}
\newtext{
In this section we present our nonlinear model of opinion dynamics for a network of interacting agents that form opinions about an arbitrary number of options. In Section~\ref{sec:linear_models} we recall the classical consensus model of DeGroot \cite{DeGroot1974} and several of the extensions that have been proposed and studied in the literature. All of the cited models (with one exception noted) use an opinion update rule that depends on a linear weighted-average of exchanged opinions. In our model, as discussed in Section~\ref{sec:properties} and formalized in Section~\ref{sec:model}, we apply a saturation function to opinion exchanges, which makes the update rule fundamentally nonlinear, even before introducing extensions. The fundamentally nonlinear update rule makes all the difference with respect to generality and flexibility of the model as we show here and in the rest of the paper. 

\subsection{Linear Averaging Models: Drawbacks and 
Extensions \label{sec:linear_models}}

Opinion formation is classically modeled as a weighted-averaging process, as originally introduced by DeGroot \cite{DeGroot1974}. In this framework an agent's opinion $x_i\in\R$ reflects how strongly the agent supports an issue or topic of interest. The real-valued opinion is updated in discrete time as a weighted average of the agent's own and other agents' opinions, i.e.,
\begin{equation}\label{EQ: DeGroot}
x_i(T+1)=a_{i1}x_{1}(T)+\cdots+a_{i\Na}x_{\Na}(T)
\end{equation}
where $a_{i1}+\cdots+a_{i\Na}=1$ and $a_{ik} \geq 0$. The  weights $a_{ik}$ describe the influence of the opinion of agent $k$ on the opinion of agent $i$ and the matrix $A\in\R^{\Na\times\Na}$ with entries $a_{ik}$ represents the structure of the influence network.

A key drawback of linear weighted-average models is that consensus among the agents is the only possible outcome.
As observed in \cite{mei2020rethinking}, this necessarily happens because the attraction strength of agent $i$'s opinion toward agent $k$'s opinion increases linearly with the difference of opinions between the two agents. In other words, the more divergent the two agents' opinions are, the more strongly they are attracted to each other, which is paradoxical from an opinion formation perspective. 

To overcome these limitations, a number of prominent variations on averaging models have been proposed. For example in ``bounded confidence" models, agents average network opinions but delete communication links to any neighbors whose opinions are sufficiently divergent from their own \cite{Deffuant2000,HegselmannKrause2002,hegselmann_opinion_2005, BlondelTAC2009}.  In a similar spirit, ``biased assimilation" models instead incorporate a self-feedback into the interaction weights of an averaging model \cite{Dandekar2013,xia2020analysis}. This self-feedback accounts for an individual's bias towards evidence that conforms with its existing beliefs. The linear model and its variations have also been extended to the case of signed networks, where the linear weights $a_{ik}$ can be negative \cite{Altafini2013,Liu2017,shi2019dynamics}. Meanwhile, in \cite{mei2020rethinking} the authors do away with averaging altogether and instead propose that opinions form through a weighted-median mechanism. 

In the present paper we propose an alternative perspective to this literature: driven by the above motivation and the model-independent theory developed in \cite{AF-MG-AB-NEL:20}, we introduce a parsimonious nonlinear extension of linear weighted-average opinion dynamics that leverages the saturation function. 

The linear weighted-average discrete-time opinion dynamics \eqref{EQ: DeGroot} can equivalently be written as
\[
x_i(T+1)=x_i(T)+\Big(-x_i(T)+a_{i1}x_{1}(T)+\cdots+a_{i\Na}x_{\Na}(T)\Big)\,.
\]
This discrete-time update rule is the unit time-step Euler discretization of the continuous time linear dynamics
\begin{equation}\label{EQ: linear consensus}
\dot x_i=-x_i+a_{i1}x_1+\cdots+a_{i\Na}x_{\Na}.
\end{equation}
Observe that~\eqref{EQ: DeGroot} and~\eqref{EQ: linear consensus} have exactly the same steady states with the same (neutral) stability.

The linear consensus dynamics~\eqref{EQ: linear consensus} are determined by two terms: a weighted-average opinion-exchange term, modeling the pull felt by agent $i$ toward the weighted group opinion, and a linear damping term, which can be interpreted as the agent's {\it resistance} to changing its opinion.

\subsection{Nonlinear Multi-option Extension of Weighted-average Models: Defining Properties\label{sec:properties}}

Our goal is to derive a novel nonlinear extension of~\eqref{EQ: linear consensus} satisfying the following defining properties.\\
{\bf 1. Opinion exchanges are saturated}.  Saturated nonlinearities appear in virtually every natural and artificial signaling network due to bounds on action and sensing. For example dynamics that evolve according to saturating interactions appear in spatially localized and extended neuronal population models of thalamo-cortical dynamics~\cite{wilsoncowan1972,wilsoncowan1973}, in Hopfield neural network models \cite{JJH:82,JJH:84,nakamura_neural-network_1995}, in 
models of perceptual decision making \cite{UM2001,Bogacz2007}, and in control systems with sensor and actuator saturations \cite{liu1994dynamical,hu2001control}. Saturated interactions between decision-makers also effectively bound the attraction between opinions, thus overcoming the linear weighted-average model paradox mentioned above.\\ 
{\bf 2. Multi-option opinion formation}. Allowing for an arbitrary number of options makes the model relevant to a wide range of applications, for example, in task allocation problems where options represent tasks or in strategic settings where options represent strategies.  We extend the model to multiple options by suitably generalizing the agent's opinion state space, analogous to existing multi-option extensions of averaging models such as \cite{Fortunato2005,Nedic2012,LiScaglione2013,friedkin_network_2016,Parsegov2017,Pan2019,ye_consensus_2019}.

To construct this extension formally, observe that in the scalar opinion setting, $x_i>0$ ($x_i<0$) is usually interpreted as favoring (disfavoring) an option A and disfavoring (favoring) an option B. The strength of favoring or disfavoring is represented by the magnitude $|x_i|$ and $x_i=0$ is interpreted as being neutral. This formalism is equivalent to one in which each agent is characterized by two scalar variables $z_{iA}$ (modeling the preference of agent $i$ for option A) and $z_{iB}$ (modeling the preference of agent $i$ for option B) that are ``mutually-exclusive'', i.e., that satisfy $z_{iA}+z_{iB}=0$. The scalar opinion is then obtained simply by defining $x_i=z_{iA}$. 
 This observation leads to the following multi-option generalization of the state space of model~\eqref{EQ: linear consensus}. Given $\No$ options, we model each agent's opinion state space as the subspace $\mathbf{1}_\No^{\perp}\subset\R^{\No}$.
Thus, in our model, the opinion state of agent $i$, $i=1,\ldots,\Na$, is described by the state variable $\Zz_i\in \mathbf{1}_\No^{\perp}$, with components $z_{ij}$,  $j=1,\ldots,\No$. When $\Zz_i=0$, we say that agent $i$ is {\it neutral} or {\it unopinionated}. When $\Zz_i\neq 0$ we say that the agent is {\em opinionated}. The full model state space is
$V=\underbrace{\mathbf{1}_\No^{\perp}\times\cdots\times \mathbf{1}_\No^{\perp}}_{\Na\text{ times}}$,
and $\Zz = (\Zz_1, \ldots, \Zz_{\Na}) \in V$ is the system state. The origin $\Zz = 0$ is the {\em neutral point}.  Another way of interpreting our choice of $\mathbf{1}_\No^{\perp}$ as an agent's state space  comes from observing that $\mathbf{1}_\No^{\perp}$ is the tangent space to the $(\No-1)$-dimensional simplex in $\R^{\No}$. Because $\mathbf{1}_\No^{\perp}$ and the simplex are isomorphic, our modeling approach naturally applies to multi-option decision-making problems in which an agent's state space is the $(\No-1)$-dimensional simplex. This is useful when the agents' opinions are interpreted as probabilities of choosing options, for example, in the case of mixed strategies in games where an option refers to a strategy \cite{ParkAllerton2021}. 
For more details on the connection to simplex dynamics see Appendix~\ref{app:simplex}.\\
\noindent {\bf 3. Agents have allocable attention}. Because an agent's attention or susceptibility to exchanged opinions may be variable, we introduce, for each agent $i$,  two parameters, $d_i>0$ and $u_i\geq0$, that weight the relative influence of the linear resistance term and the opinion-exchange term, respectively. When the {\it resistance parameter} $d_i$ dominates the {\it attention parameter} $u_i$, the agent is weakly attentive to other agents' opinions. When $u_i$ dominates $d_i$, the agent is strongly attentive to other agents' opinions. A shift from a weakly attentive to a strongly attentive state can be induced, for instance, by a time-urgency (election day approaching) or a spatial-urgency (target getting closer) to form an informed collective opinion. The attention parameter $u_i$ can also be used to model {\em social effort}, {\em excitability}, or {\em susceptibility} of agent $i$ to social influence.\\
{\bf 4. Agents have exogenous inputs}.  For each agent, we introduce an input parameter $b_{ij}$, which	represents an input signal from the environment or a bias or predisposition that directly affects agent $i$'s opinion of option $j$. For example, the input $b_{ij}$ can be used to model the exogenous influence of agent $i$'s initial opinions, as in  \cite{FriedkinJohnsen1999}, where agents hold on to their initial opinions (sometimes called ``stubborn" agents as in \cite{GhaderiSrikant2014}). 

If the attention and/or bias parameters are hard or impossible to measure or control, which may be the case in sociopolitical applications, we can use standard homogeneity assumptions, e.g., $d_i=1$, $u_i=u$, $b_{ij}=0$ for all agents, and include random perturbations to capture modeling uncertainties. 
In technological applications (e.g. robotic swarms), however, tunable parameters of the model provide novel, analytically tractable means to design complex collective behaviors - see for example \cite{franci2021analysis}. }

\subsection{A General Nonlinear Opinion Dynamics Model}
\label{sec:model}

\begin{figure}
    \centering
        \includegraphics[width=0.3\textwidth]{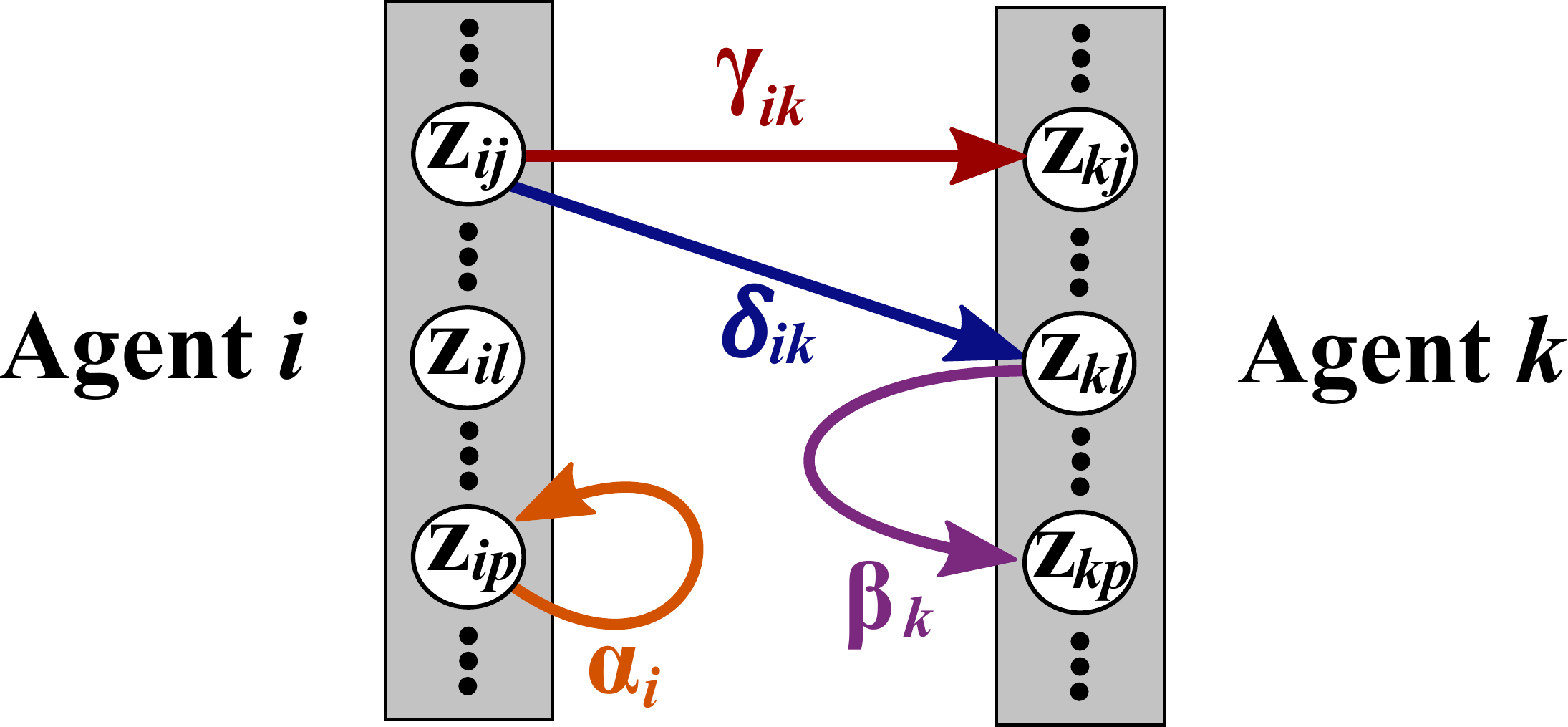}
    \caption{Illustration of the four classes of interactions. \newtext{An arrow from $z_{ij}$ to $z_{kl}$ means the opinion of agent $k$ about option $l$ is influenced by the opinion of agent $i$ about option $j$, modulo the labeled gain. }}
    \label{fig:network} 
\end{figure}

\newtext{In the multi-option setting, there are four possible types of coupling in the resulting opinion network (see Figure \ref{fig:network}):} \\
1) \textit{Intra-agent, same-option coupling}, with gain \newtext{$\alpha_i$};\\
2) \textit{Intra-agent, inter-option coupling}, with gain \newtext{$\beta_i$};\\
3) \textit{Inter-agent, same-option coupling}, with gains \newtext{$\gamma_{ik}$, $i \neq k$ };\\
4) \textit{Inter-agent, inter-option coupling}, with gains \newtext{ $\delta_{ik}$, $i \neq k$. }\\
Parameters $\alpha_i$, $\beta_i$, $\gamma_{ik}$, $\delta_{ik}$ determine  qualitative properties of  opinion interactions. Parameter $\alpha_i$ determines  sign and magnitude of opinion self-interaction for agent $i$. To avoid redundancy with resistance  $d_i$, we assume $\alpha_i\geq0$, i.e., either no self-coupling ($\alpha_i=0$) or {\em self-reinforcing} coupling ($\alpha_i>0$). Parameter $\beta_i$ determines how different intra-agent opinions interact.
Parameters $\gamma_{ik}$ and $\delta_{ik}$ determine whether agents $i$ and $k$ {\em cooperate} ($\gamma_{ik}-\delta_{ik}>0$) or {\em compete} ($\gamma_{ik}-\delta_{ik}<0)$. When different option dimensions have no interdependence, we can set $\beta_{i} = \delta_{ik} = 0$ for all $i,k = 1, \dots, \Na$. 
\newtext{The proposed general nonlinear opinion dynamics are
\begin{subequations}\label{EQ:generic decision dynamics}
	\begin{align}
	& \quad \; \; \;\dot \Zz_i=P_0{\bs F}_i(\Zz)\\
	&F_{ij}(\Zz)=-d_{i}z_{ij} +b_{ij}+  
	u_{i}\left(  S_{1}\left( {\alpha_i z_{ij} + \textstyle\sum_{\substack{k\neq i\\k=1}}^\Na \gamma_{ik} z_{kj}}\right) \right. \nonumber \\ &
	\left.+ \textstyle\sum_{\substack{l\neq j\\l=1}}^\No S_{2}\left( {\beta_i z_{il} + \textstyle\sum_{\substack{k\neq i\\k=1}}^\Na \delta_{ik} z_{kl}}\right)\right). \label{eq:modelb} 
	\end{align}
\end{subequations}}
$S_{q}: \R \to [-k_{q1},k_{q2}]$ with $k_{q1},k_{q2} \in \R^{>0}$ 
for $q \in \{1,2\}$ is a generic sigmoidal saturating function satisfying constraints  $S_{q}(0) = 0$, $S_{q} '(0) = 1$, $S_{q} ''(0) \neq 0 $, $S_{q}'''(0) \neq 0 $. $S_{1}$ saturates same-option interactions, and $S_{2}$ saturates inter-option interactions. $S_1$ and $S_2$ could be the same but are distinguished  in \eqref{EQ:generic decision dynamics} for a more general statement of the model.\newtext{ We provide an even more general formulation of the model in Appendix~\ref{app:adjacencytensor} that makes use of an adjacency tensor and allows for the possibility of heterogeneous interactions between options.
In the following we let
\begin{equation}\label{eq: gamma delta matrices}
    \Gamma=[\gamma_{ik}]\in\R^{\Na\times\Na},\quad \Delta=[\delta_{ik}]\in\R^{\Na\times\Na}.
\end{equation}
We note that in~\eqref{EQ:generic decision dynamics} the sum over the agents could be brought outside of the two sigmoids without altering the qualitative behavior of the model. 
Our choice in~\eqref{EQ:generic decision dynamics}  corresponds to an opinion network with saturated inputs. Bringing the sum over the agents outside the sigmoids  corresponds to an opinion network with saturated outputs. Either choice could be useful depending on the application. On the other hand, the sum over the options cannot be brought inside $S_2$ as the mutual exclusivity condition $\Zz_i\in \mathbf{1}_\No^{\perp}$ would lead to spurious term cancellations for some parameter choices. Intuitively, this means that opinions about different options are processed though different input channels. Dynamics~\eqref{EQ:generic decision dynamics} are well defined on the system state space $V$, as we rigorously prove in Appendix \ref{app:simplex}. }

Let $\hat{b}_{i} = \frac{1}{\No} \sum_{l = 1}^{\No} b_{il}$ be the {\em average input to agent $i$} and let $b_{ij}^\perp = b_{ij} - \hat b_i$ be the {\em relative input to agent $i$ for option $j$}.

\newtext{\begin{lemma}\label{lem:relative_bias}The dynamics \eqref{EQ:generic decision dynamics} are independent of the average input $\hat{b}_{i}$ in the sense that $\frac{\partial \dot{z}_{ij}}{\partial \hat{b}_{i}} = 0$. 
\end{lemma}
\begin{proof}
Recall that $P_0$ is the projection onto $\mathbf{1}_\No^{\perp}$ as defined in Section~\ref{sec:notation}. Then $P_0 \mathbf{b}_i = \textbf{b}_{i}^{\perp}$, and the conclusion follows trivially from the form of \eqref{EQ:generic decision dynamics}.
\end{proof}}
Lemma~\ref{lem:relative_bias} implies that only  relative inputs affect the location of the equilibria of the opinion dynamics \eqref{EQ:generic decision dynamics}.  
\begin{assumption} \label{assump:bperp}
In light of Lemma \ref{lem:relative_bias}, for the remainder of the paper we assume without loss of generality that the average input $\hat{b}_i = 0$ for all $i = 1, \dots, \Na$. Thus, $b_{ij} = b_{ij}^{\perp}$.  
\end{assumption}

Without relative inputs, the system \eqref{EQ:generic decision dynamics} always has the neutral point as an equilibrium.

\begin{lemma} \label{lem:zeroinput}
$\Zz = 0$ is an equilibrium for \eqref{EQ:generic decision dynamics} if and only if there are no relative inputs, i.e., $b_{ij} = 0$ for all $i$  and all $j$.
\end{lemma}

When relative inputs are small, i.e., they do not dominate the dynamics, the formation of opinions in the general model \eqref{EQ:generic decision dynamics} is governed by the balance between the resistance term, which inhibits opinion formation, and the social term, which promotes opinion formation.  For illustrative purposes, consider the case in which $u_i = u \geq 0$ for all $i$. Then for $u$ small, resistance dominates and the system behaves linearly. The opinions $z_{ij}$ remain small and their relative magnitude is determined by the small inputs $b_{ij}$. For $u$ large, the social term dominates and the system behaves nonlinearly. 

\newtext{Importantly, in the nonlinear regime, opinions $z_{ij}$ form that are much larger than, and potentially unrelated to, inputs $b_{ij}$, even for very small initial conditions. Opinion exchanges govern opinion formation through bifurcation mechanisms as discussed in the next section and formalized and investigated in the remainder of the paper.}

\newtext{\subsection{Generality and Connection to Existing Models 
}\label{sec: generalizing specializing}

The model~\eqref{EQ:generic decision dynamics} is general in the sense that it recovers a number of published opinion-formation, decision-making, and consensus models for specific sets of parameters and/or when linearized. In order to illustrate this we consider the model specialized to $\No = 2$, as most of the models in the literature consider two-option scenarios. The opinion state of agent $i$ is one-dimensional: following the notation introduced in Section \ref{sec:properties} we define $x_i = z_{i1} = -z_{i2}$ as agent $i$'s opinion.  Then, opinion dynamics \eqref{EQ:generic decision dynamics} reduce to 
\begin{multline} \label{eq:Nby2}
    \dot{x}_{i} = - d_{i} x_{i} + u_{i} \left( \hat{S}_1 \left( {\alpha_i x_{i} + \textstyle\sum_{\substack{k\neq i\\k=1}}^\Na \gamma_{ik} x_{k}}\right) \right.  \\ 
	\left. - \hat{S}_{2}\left( {\beta_i x_{i} + \textstyle\sum_{\substack{k\neq i\\k=1}}^\Na \delta_{ik} x_{k}}\right)\right)+b_{i} 
\end{multline}
where $\hat{S}_l(y) = \frac{1}{2}(S_l(y) - S_l(-y))$ are odd saturating functions for $l = 1,2$, $b_{i} := b_{i1} = - b_{i2}$, and $d_{i} = \frac{1}{2} (d_{i1} + d_{i2})$. Let the network opinion state be $\mathbf{x} = (x_{1}, \dots, x_{\Na}) \in \mathds{R}^{\Na}$ and vector of inputs be $\mathbf{b} = (b_{1}, \dots, b_{\Na})\in \mathds{R}^{\Na}$. When interactions between option dimensions are disregarded, i.e. with $\beta_{i} = \delta_{ik} = 0$ for all $i,k = 1, \dots, \Na$, the two-option model \eqref{eq:Nby2} further reduces to
\begin{equation}
    \dot{x}_{i} = - d_{i} x_{i} + u_{i} \hat{S}_1 \left( {\alpha_i x_{i} + \textstyle\sum_{\substack{k\neq i\\k=1}}^\Na \gamma_{ik} x_{k}}\right) +b_{i} \label{eq:Nby2_shortened}
\end{equation}
which, with appropriate restrictions on the model coefficients, recovers a number of nonlinear consensus models studied in recent literature. We illustrate this in the following example. 

\begin{example}[Specialization to nonlinear consensus protocols] \label{ex:specialization}
A. When $\alpha_i = 0$, $\gamma_{ik} \in \{0,1\}$ (or more generally, $\gamma_{ik} \geq 0$), $u_i := u \geq 0$, and the resistance parameter $d_i$ is defined as  $d_i := \sum_{k = 1}^{\Na} \gamma_{ik} $ with $k \neq i$ (the network in-degree for node $i$), \eqref{eq:Nby2_shortened} reduces to the nonlinear consensus dynamics of \cite{AF-VS-NEL:15a,Gray2018,fontan2017multiequilibria,Abara2018}. 

B. When $\alpha_i = 0$, $\gamma_{ik} \in \{0,1,-1\}$ (or more generally, $\gamma_{ik} \in \mathds{R}$), $u_i := u \geq 0$, and the resistance parameter $d_i$ is defined as $d_i := \sum_{k = 1}^{\Na} |\gamma_{ik}| $ with $k \neq i$, \eqref{eq:Nby2_shortened} reduces to the nonlinear consensus dynamics with antagonistic interactions studied in \cite{Fontan2018,fontan2021role}. 
\end{example}

In the nonlinear consensus models of Example \ref{ex:specialization}, the formation of consensus opinions on the network is a bifurcation phenomenon. Namely when $b_{i} = 0$ for all $i = 1, \dots, \Na$ and $0 \leq u < u^{*} \leq 1$, the 
neutral point $\mathbf{x} = 0$ is an asymptotically stable equilibrium. At a critical value $u = u^*>0$ a pitchfork bifurcation is observed in both models, at which point $\mathbf{x} = 0$ loses stability and two non-zero asymptotically stable equilibria appear \cite[Theorem 1]{Gray2018}, \cite[Theorem 1]{fontan2021role}. For nonzero inputs, the pitchfork unfolds.

Importantly, the linearization of these models about the origin $\mathbf{x}=0$ at $u = 1$ yields $\dot{\mathbf{x}} = - (D - \Gamma) \mathbf{x}$, where $D= \operatorname{diag}(d_{i}) \in \R^{\Na\times\Na}$ is the degree matrix for the network.
For the positive weights of Example \ref{ex:specialization}.A this corresponds to the standard Laplacian consensus protocol \cite{OlfatiSaber2004}, a continuous-time analogue of the weighted-average models discussed in Section \ref{sec:linear_models}. For the signed weights of Example \ref{ex:specialization}.B this linearization is exactly the model of linear consensus with antagonistic, i.e., signed, interconnections \cite{Altafini2013,Liu2017,shi2019dynamics}, which is sometimes referred to as the ``Altafini" model.

In linear models, nonzero agreement (consensus) and disagreement (e.g., bipartite consensus and its generalizations) equilibria are never exponentially asymptotically stable because the model Jacobian has a zero eigenvalue. The eigenspace of the zero eigenvalue is $\R\{\bs{1}\}$ in the case of agreement, whereas it is spanned by a mixed-sign vector determined by the coupling topology in the case of disagreement~\cite{Altafini2013,Pan2019,meng2016interval,proskurnikov2015opinion}. In other words, linear agreement and disagreement models are not structurally stable and arbitrary small unmodelled (nonlinear) dynamics will in general destroy the predicted behavior. 
Adding saturated opinion exchanges has a two-fold advantage: i) it makes the model generically structurally stable and, therefore, the agreement and disagreement equilibria hyperbolic (i.e., with no eigenvalues on the imaginary axis); ii) it weakens the necessary conditions for the existence of stable disagreement states. In linear models, the existence of neutrally stable agreement or disagreement states is always linked to restrictive and non-generic assumptions on the coupling topology, for example, balanced coupling for consensus~\cite{OlfatiSaber2004} and either strongly connected structurally balanced coupling~\cite{Altafini2013,Pan2019}, quasi-strongly connected coupling with an in-isolated structurally balanced subgraph~\cite{proskurnikov2015opinion}, or the existence of a spanning tree on the coupling graph~\cite{meng2016interval} for disagreement. As we  rigorously show in the next section, in our model agreement is always possible for generic strongly connected (balanced or unbalanced) graphs, whereas disagreement only requires a weak and provable condition on the spectral properties of the adjacency matrix. It follows that our model recovers the behavior of linear models when one of the above conditions is satisfied (Figure~\ref{fig:AltafiniCompare}) but also highlights the conservativeness of linear model predictions under more general coupling topologies (Figure~\ref{fig:AltafiniCompare2}). }

\begin{figure}
    \centering
    \includegraphics[width=0.85\columnwidth]{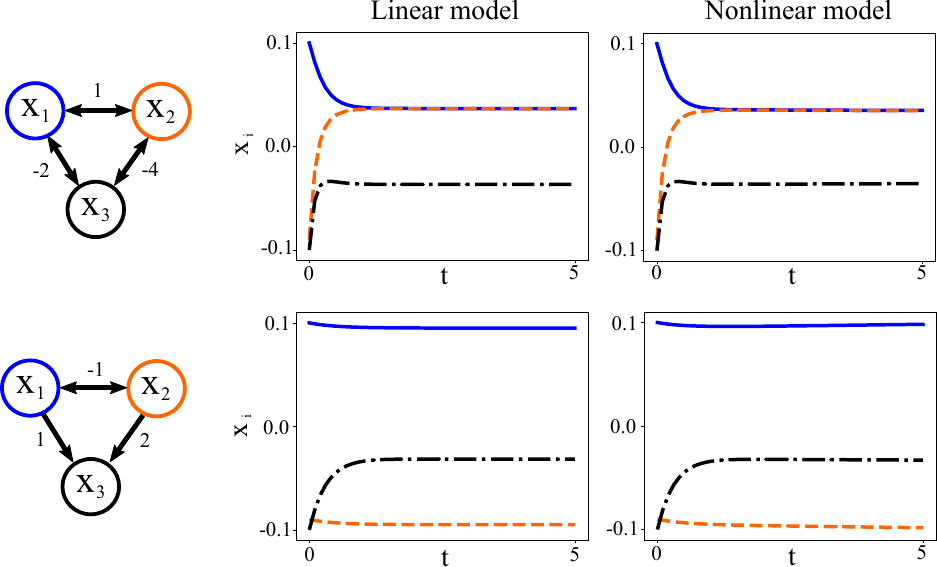}
    \caption{For small initial conditions $\mathbf{x}(0) = (0.1,-0.09,-0.1)$, trajectories of the linear model from \cite{Altafini2013} (left) approximate trajectories of the nonlinear model of Example \ref{ex:specialization}.B (right) with $\Na=3$, $u = 1.01$, $b_{i} = 0$, and $\hat{S}_{1} = \tanh$. Top: bipartite consensus on a strongly connected and structurally balanced graph, as in \cite[Example 1]{Altafini2013}. \newtext{Bottom: polarized opinions on a quasi-strongly connected graph containing an in-isolated structurally balanced subgraph, as in  \cite[Example 1]{proskurnikov2015opinion}.}  } 
    \label{fig:AltafiniCompare}
\end{figure}

\begin{figure}
    \centering
    \includegraphics[width=0.85\columnwidth]{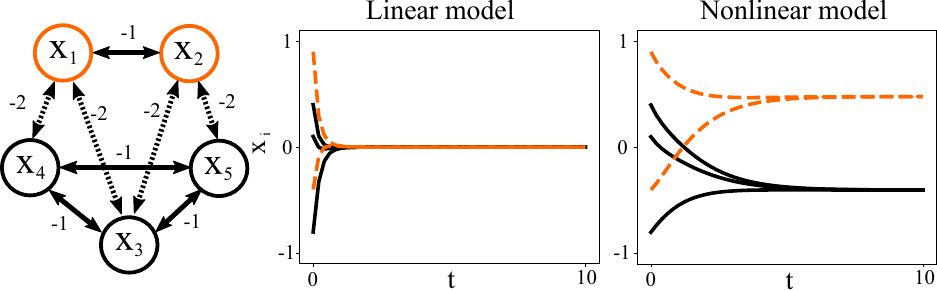}
    \caption{The model from \cite{Altafini2013} (left) and nonlinear model  (\ref{eq:Nby2_shortened}) (right) with $\Na=5$, $u_i = 0.5$, $d_i = 1$, $b_{i} = 0$ for all $i = 1, \dots, \Na$, $\hat{S}_{1} = \tanh$, initial conditions $\mathbf{x}(0) = (0.9,-0.4,0.4,0.1,-0.8)$, and  same adjacency matrix given by $\gamma_{ik} = -1$ for $i,k \in \mathcal{I}_{p}$, $i\neq k$, and $\gamma_{ik} = -2$ for $i \in \mathcal{I}_{p}$, $ k\in \mathcal{I}_{s}$, $p \neq s$ for clusters with indices $\mathcal{I}_{1} = \{1,2\}$ and  $\mathcal{I}_{2} = \{3,4,5\}$ - \newtext{ see network diagram  for  illustration of the interconnection topology.} The linear model converges to the neutral solution. 
   The nonlinear model, however, converges to a stable clustered dissensus state, as follows from Remark \ref{rem:clustered_bif}. }
    \label{fig:AltafiniCompare2}
\end{figure}

\subsection{Clustering and Model Reduction}\label{sec: clustering}
\newtext{
The opinion states $\Zz_i$ of the model \eqref{EQ:generic decision dynamics} can either represent individual agents or alternatively the average opinion of a subgroup. The latter perspective can be advantageous, for example, in designing methodology for robotic swarm activities where subgroups of robots needs to make consensus decisions, 
 in studying cognitive control where the behavior of competing subpopulations of neurons determines task switching \cite{MusslickBizyaevaCogSci},  
and in modeling and investigating mechanisms that explain sociopolitical processes such as political polarization \cite{Leonard2021}. In this section we prove a sufficient condition for \textit{cluster synchronization} of the opinions on the network with the opinion dynamics \eqref{EQ:generic decision dynamics}, in which the network trajectories converge to a lower-dimensional manifold on which agents within each cluster have identical opinions. 

The cluster synchronization problem has been extensively studied in dynamical systems with diffusive coupling, as in \cite{pogromsky2002partial, xia2011clustering}. More broadly, cluster synchronization has been linked to graph symmetries and graph structure called \textit{external equitable partitions} \cite{stewart2003symmetry,o2013observability,schaub2016graph,sorrentino2016complete,gambuzza2019criterion}. In the following theorem we show that such a network structure constitutes a sufficient condition for a network of agents to form opinion clusters -- see Appendix~\ref{app:clusters} for the proof. }

\begin{theorem}[Model Reduction with Opinion Clusters]\label{thm:clusters}
Consider $N_c$ clusters with $N_p$ agents in cluster $p$ such that $\sum_{p = 1}^{N_{c}} N_p = \Na$. Let $\mathcal{I}_{p}$ be the set of indices for agents in cluster $p$. Assume for every  $p = 1, \dots, N_{c}$: 1) $u_{i} = \hat u_{p}$, $d_{i} = d_{p}$, $b_{ij} = b_{pj}$ for $i \in \mathcal{I}_{p}$; 2) within a cluster $\alpha_i = \bar \alpha_p$, $\gamma_{ik} = \tilde \alpha_p$, $\beta_i =\bar \beta_p$, $\delta_{ik} = \tilde \beta_p $ for $i,k \in \mathcal{I}_{p}$, and $i \neq k$; 3) between clusters $\gamma_{ik} = \tilde \gamma_{ps}$, $\delta_{ik} = \tilde \delta_{ps}$   for $i \in \mathcal{I}_{p}$, $k \in \mathcal{I}_{s}$ $s = 1, \dots, N_{c}$ and $s \neq p$. Define bounded set $K_q\subset \mathds{R}^{>0}$, $q=1,2$, as the image of the derivative of the saturating function $S'_q$ of \eqref{EQ:generic decision dynamics}. If the following condition holds: 
\begin{equation}\label{eq:clustering_cond}
    \sup_{\kappa_{1} \in K_{1}, \kappa_{2} \in K_{2}}\Big\{- d_{p}  + u_{p} \kappa_{1} ( \bar \alpha_p - \tilde \alpha_p )  
    + u_{p} \kappa_{2} (\bar\beta_p - \tilde \beta_p  ) \Big\} < 0,
\end{equation}
for all $p = 1, \dots, N_{c}$, then every trajectory of \eqref{EQ:generic decision dynamics} 
converges exponentially to the $N_{c}(\No-1)$-dimensional manifold
\begin{equation}
    \mathcal{E} = \{ \Zz \in V \ | \ z_{ij} = z_{kj} \ \ \forall i,k \in \mathcal{I}_{p}, \ p = 1, \dots, N_{c} \}. \label{eq:ClusteredManifold}
\end{equation}
The dynamics on $\mathcal{E}$ reduce to (\ref{EQ:generic decision dynamics}) with $N_{c}$ agents with opinion states $\hat{z}_{pj}$, $p = 1, \dots, N_{c}$, and with coupling 
weights 
\begin{subequations}\label{eq:cluster_weights}
\begin{align}
    \hat{\alpha}_{p} &= \bar{\alpha}_{p} + (N_{p} - 1) \tilde \alpha_p, \ \ &\hat{\gamma}_{ps} = N_{s} \tilde \gamma_{ps}, \\
    \hat{\beta}_{p} &= \bar{\beta}_{p} + (N_{p} - 1) \tilde \beta_{p}, \ \ &\hat{\delta}_{ps} = N_{s} \tilde \delta_{ps}.
    \end{align}
\end{subequations}
\end{theorem}

Whenever conditions of Theorem \ref{thm:clusters} are met, the group of $\Na$ agents will converge to a clustered group opinion state.  This can happen for a broad class of interaction networks 
including an all-to-all network with interaction weights that all have the same sign. The sufficient condition can, for example, inform network design for technological systems where several groups of units must make collaborative decisions. See Figure~\ref{fig:AltafiniCompare2} for an illustration of opinion trajectories with two clusters, membership in which is defined by the network structure. 

\newtext{
\subsection{A Minimal Opinion Network Model}
\label{SSEC: minimal model}

Several of the results characterizing opinion formation in ~\eqref{EQ:generic decision dynamics} will be proved in the {\it homogeneous regime} defined by}
\begin{align}\label{eq: homogeneous params}
b_{ij}&=0,\ d_{i} = d >0,\ u_{i} = u\geq 0,\ \alpha_i = \alpha \in \R\nonumber\\
\beta_i &= \beta \in \R,\ \gamma_{ik} = \gamma a_{ik},\ \delta_{ik} = a_{ik}, \ A = [a_{ik}]
\end{align}
where $\alpha,\beta,\gamma, \delta \in \R$ and $a_{ik} \in \{ 0, 1 \}$, $a_{ii} = 0$ for all $i,k = 1,\ldots,\Na$, $k \neq i$, so that $A$ is an unweighted adjacency matrix without self-loops.

\newtext{With this choice of parameters, the nonlinear model is minimal in the following sense. The matrix $A$ with elements $a_{ik}$ defines the influence network topology. The set of four interactions gains $\alpha,\beta,\gamma,\delta$ is minimal because in general there are four distinct types of arrows in a multi-option opinion network. Removing one of these four parameters will in general compromise the existence of possibly crucial opinion-formation behaviors. The (global) attention parameter $u$ and 
resistance parameter $d$ tune an agent's attention to other agents' opinions and they jointly determine the occurrence of opinion-formation bifurcations, as we  prove in Section \ref{sec:agree_disagree}.

We show that our model, even in the fully homogeneous regime, exhibits extremely rich and analytically provable opinion-formation behaviors. We further  build upon the results proved for the homogeneous model to study, either analytically or numerically, the effects of heterogeneity and perturbations.
}

\section{Agreement and disagreement opinion formation \label{sec:agree_disagree}}

In this section, we show the following key results on opinion formation for dynamics~\eqref{EQ:generic decision dynamics}.

\begin{enumerate}
\item Opinion formation can be modeled as a bifurcation, an intrinsically nonlinear dynamical phenomenon. \newtext{Opinions form rapidly through bifurcation-induced instabilities rather than slow linear integration of evidence.} Opinions can form even in the absence of input, as long as attention (urgency or susceptibility, etc.) is sufficiently high.
\item \newtext{The way opinions form at a bifurcation depends on the  eigenstructure of the matrix $\Gamma-\Delta$ defined by \eqref{eq: gamma delta matrices}.
\item In the homogenous regime defined by \eqref{eq: homogeneous params}, cooperative agents ($\gamma>\delta$) always form agreement opinions, whereas under suitable assumptions on the eigenstructure of the adjacency matrix $A$ competitive agents ($\gamma<\delta$) always form disagreement opinions.}
\item At the bifurcation, there are multiple stable solutions, and opinion formation breaks deadlock, \newtext{that is, the situation in which all agents remains neutral, and therefore undecided, about all the options.}
\item Near the bifurcation, opinion formation is ultra-sensitive to input. \newtext{
\item Away from the bifurcation, opinion formation is robust to small heterogeneity in parameter values and small inputs.  
\item  In the absence of inputs, multistable agreement solutions  {\it and} multistable  disagreement solutions emerge generically at opinion-forming bifurcations.
\item In the presence of inputs, the opinion-forming bifurcation unfolds (i.e., multistability is broken) in a such a way that the opinion states favored by inputs attract most of the initial conditions close to the bifurcation. The network structure governs the relative influence of inputs, which leads to a formal notion of centrality indices for agreement and disagreement.}
\item \newtext{ Agreement and disagreement} can {\it co-exist}, revealing the possibility of easy transition between agreement and disagreement.
\item \newtext{With sufficient symmetry, agreement specializes to consensus and disagreement to dissensus.}
\end{enumerate}
\newtext{
\subsection{Agreement and Disagreement 
States\label{SSEC: agreement disagreement states}}

In this section we clarify what it means for agents in a group to agree and disagree. We say  the agents \textit{agree}, i.e., are in an {\em agreement state}, when $\operatorname{sign}(z_{ij}) = \operatorname{sign}(z_{kj})$ for all $i,k = 1, \dots, \Na$, $j = 1, \dots, \No$. This means that all agents unanimously favor or disfavor each of the options, although they may differ on the magnitude of their opinions. 
Agreement specializes to \textit{consensus} when $\Zz_i = \Zz_k$ for all $i,k = 1, \dots, \Na$. 
We say the agents \textit{disagree}, i.e., are in a {\em disagreement state}, when $\operatorname{sign}(z_{ij}) \neq \operatorname{sign}(z_{kj})$ for at least one pair of agents $i,k = 1, \dots, \Na$, $i \neq k$, and at least one option $j$.  Disagreement specializes to \textit{dissensus}  when the average opinion of the group is neutral, i.e., $\sum_{i=1}^{\Na} \Zz_i = 0$. 

}



\newtext{
\begin{remark}\label{RMK: threshold}
    In the presence of nonzero inputs $b_{ij}$,  agents will generically have nonzero opinions about options as follows from Lemma \ref{lem:zeroinput}. For realistic applications, small opinions formed in a linear response to inputs should be distinguished from large opinions which arise from a nonlinear response. To make this distinction we  say agents are {\bf opinionated} when their opinions are large, and {\bf unopinionated} when their opinions are close to zero. In this paper we keep this distinction qualitative. A precise bound between opinionated and unopinionated  magnitudes depends on the application and can be defined when necessary. 
\end{remark}
}

\newtext{
\subsection{Opinions Form through a Bifurcation\label{sec:bifurcations}}

In this section we prove how 
steady-state bifurcations of the opinion dynamics \eqref{EQ:generic decision dynamics} result in nonzero opinions on the network.
The following theorem, proved in Appendix~\ref{app:mainTheorem}, provides sufficient conditions under which 
{opinions form through} a bifurcation from the neutral equilibrium $\Zz=0$ and  formulas to compute the kernel along which the bifurcation appears. 

\begin{theorem}[Opinion Formation as a Bifurcation] \label{thm:Bifurcations}
Consider model~(\ref{EQ:generic decision dynamics}) with $b_{ij} = 0$, $d_{i} = d$, $u_i = u$, $\alpha_i = \alpha$, and $\beta_i = \beta$, for all $i = 1, \dots, \Na$.
Let $J$ be the Jacobian of the system evaluated at neutral equilibrium $\Zz = 0$.   Define $\lambda$ to be the eigenvalue of $\Gamma - \Delta$ with largest real part, with $\Gamma$ and $\Delta$ from~\eqref{eq: gamma delta matrices}. Assume that $\lambda$ is real, $\alpha-\beta+\lambda > 0$, and that $\operatorname{Re}[\mu] \neq \lambda$ for any eigenvalue $\mu \neq \lambda$ of $\Gamma - \Delta$. Then $\Zz = 0$ is locally exponentially stable for $0 < u < u^*$, with 
\begin{equation}
    u^* = \frac{d}{\alpha - \beta + \lambda},
\end{equation}
and unstable for $u > u^*$. If $\lambda$ is simple\footnote{\newtext{This result can be generalized to networks for which $\lambda$ is not simple, which we leave for future work.} }, at $u = u^*$ an opinion-forming steady-state bifurcation happens along 
$\ker J=\R\{\mathbf{v}^* \} \otimes \mathbf{1}_{\No}^{\perp}$ where $\mathbf{v}^*$ is the right unit eigenvector associated to $\lambda$. More precisely, generically, for each bifurcation branch there exists $\mathbf{v}_{ax}\in\mathbf{1}_{\No}^{\perp}$ such that the branch is tangent at $\Zz=0$ to the one-dimensional subspace $\mathds{R} \{ \mathbf{v}^* \otimes \mathbf{v}_{ax} \}$.
\end{theorem}

\begin{remark}\label{rmk:axial v}
    The vector $\mathbf{v}_{ax}$ can be computed as the generator of the fixed-point subspace of an {\em axial subgroup} \cite[Section~1.4]{GolubitskySymmetryPerspective} of the (irreducible) action of $\ES_{\No}$ on $\ker J$.
\end{remark}

Theorem~\ref{thm:Bifurcations} reveals how nonzero opinions can form even without input: opinions form when attention $u$ is greater than threshold $u^*$. This means that deadlock can be avoided even when there is little or no evidence to distinguish among options. The value of the threshold is determined from the structure of the communication network. Additionally, from this result we can deduce how agreement and disagreement solutions are informed by the network structure. In particular,  the equilibrium opinions of each agent near the bifurcation are directly proportional to the vector $\mathbf{v}_{ax}$, scaled by the entries of $\mathbf{v}^*$. When all of the entries of $\mathbf{v}^*$ have the same sign, the agents will be in an agreement state. 
If $\mathbf{v}^*$  contains mixed-sign entries, the agents will necessarily be in a  disagreement state. This provides a straightforward connection between the spectral properties of the effective inter-agent communication graph $\Gamma - \Delta$ and the opinion configurations which arise from the opinion dynamics \eqref{EQ:generic decision dynamics}. On the other hand, the entries of the vector $\mathbf{v}_{ax}$ determine the relative preference associated to the various options. In the following corollary we show how in the homogeneous regime \eqref{eq: homogeneous params} Theorem~\ref{thm:Bifurcations} specializes to simple conditions for agreement and disagreement. 

\begin{corollary}[Agreement and Disagreement] \label{cor:AgreeDisagree}
Consider model~(\ref{EQ:generic decision dynamics}) with homogeneous parameters as in~\eqref{eq: homogeneous params} on a strongly connected graph.  Let $\lambda_{max}>0$ be the largest real-part eigenvalue of $A$, i.e. the Perron-Frobenius eigenvalue, with associated positive eigenvector $\mathbf{v}_{max}$. Let $\lambda_{min}<0$ be the smallest real-part eigenvalue of $A$. Assume $\lambda_{min}$ is real, simple, and for all eigenvalues $\xi \neq \lambda_{min}$ of $A$, $\operatorname{Re}[\xi] \neq \lambda_{min}$.

\noindent A. \textbf{Cooperative agents.} Suppose that $\gamma-\delta>0$ and that $\alpha - \beta + \lambda_{max}(\gamma - \delta)>0$.  Then the critical value of attention at which the steady-state bifurcation predicted by Theorem~\ref{thm:Bifurcations} happens is given by
\begin{equation}
    u^* := u_a = \frac{d}{\alpha - \beta + \lambda_{max}(\gamma - \delta)} \label{eq:uc}
\end{equation}
and close to bifurcation all the bifurcation branches are made of agreement solutions. 

\noindent B. \textbf{Competitive agents.} Suppose $\gamma-\delta<0$ and that $\alpha - \beta + \lambda_{min}(\gamma - \delta) > 0$. Then the critical value of attention at which the steady-state bifurcation predicted by Theorem~\ref{thm:Bifurcations} happens is given by
\begin{equation}
    u^* := u_d = \frac{d}{\alpha - \beta + \lambda_{min}(\gamma - \delta)}. \label{eq:ud}
\end{equation}
Moreover, whenever $\mathbf{v}_{min}$, the eigenvector associated to $\lambda_{min}$, has mixed-sign entries, close to bifurcation all the bifurcation branches are made of disagreement solutions. 
\end{corollary}

We emphasize that the assumption about eigenvalues $\lambda$ of Theorem \ref{thm:Bifurcations} and $\lambda_{min}$ of Corollary \ref{cor:AgreeDisagree} being simple often holds, and can be easily verified numerically for various graph structures. Furthermore, the eignevector $\mathbf{v}_{min}$ of Corollary~\ref{cor:AgreeDisagree} typically has mixed-sign entries, and competition between agents therefore tends to result in network disagreement. For example, on undirected networks $\mathbf{v}_{min}$ always has mixed-sign entries since $\mathbf{v}_{max}$ is the positive Perron-Frobenius eigenvector and $\langle \mathbf{v}_{max}, \mathbf{v}_{min} \rangle = 0$. For example see Figure~\ref{fig:GraphClasses} for patterns of agreement and disagreement solutions for $\No = 2$ and several representative undirected graphs.}

An important feature of the opinion dynamics \eqref{EQ:generic decision dynamics} is the multistability of opinion configurations at the bifurcations described by Theorem~\ref{thm:Bifurcations} and Corollary~\ref{cor:AgreeDisagree}.  When agents cooperate and $\ker J$ is made of agreement vectors, if agreement in favor of one option is stable then agreement in favor of each other option is stable, and likewise for  disagreement solutions. There is a deadlock when $u<u_c$ ($u< u_d$) and breaking of deadlock when $u > u_c$ ($u  > u_d$).

At the bifurcation the linearization is singular, and the model is ultra-sensitive at  transition from neutral to opinionated. Even infinitesimal perturbations (e.g., tiny difference in option values) are sufficient to destroy multistability at bifurcation by selecting a subset of  stable equilibria (e.g., those corresponding to higher-valued options), a phenomenon known as forced-symmetry breaking and widely exploited in nonlinear decision-making model~\cite{Pais2013,Reina2017,Gray2018}.

Generically, stable equilibria that appear at the bifurcation are hyperbolic, and thus they and their basin of attraction are robust to perturbations, a key property that ensures stability of opinion formation despite (sufficiently small) changes in inputs,  heterogeneity in parameters, and perturbations in the communication network. Robustness bounds 
can be derived using  methods like those used for Hopfield networks in \cite{wang1994}. Robust multistability of  equilibria gives the opinion-forming process hysteresis, and thus memory, between  different opinion states: once an opinion is formed in favor of an option, a large change in the inputs is necessary for a switch.

\begin{remark} \label{rem:clustered_bif}
    Under the clustering conditions of Theorem~\ref{thm:clusters}, we can apply Theorem~\ref{thm:Bifurcations} and Corollary~\ref{cor:AgreeDisagree} with $N_c$ agents and coupling parameters defined by~\eqref{eq:cluster_weights}.
\end{remark}

\begin{remark}[Mode Interaction and Coexistence of Agreement and Diagreement] \label{rem:mode}
    When $\gamma = \delta$, there is {\em \textbf{mode interaction}} \cite{Golubitsky1985}, and agreement and disagreement bifurcations appear at the same critical value of $u$.
    This regime is especially interesting
    because it allows for co-existence of stable agreement and disagreement solutions,
    which can result in agents easily transitioning between the two in response to changing conditions. However, additional primary solution branches not captured by the analysis presented here can appear in this regime, and we leave exploring this parameter regime more thoroughly to future work.
\end{remark}

\newtext{
\subsection{Patterns of Opinion Formation for Two Options}\label{sec: patterns of opinions}

In this section we examine 
the ultra-sensitivity of the network opinion dynamics to inputs or biases of individual agents when operating near its bifurcation point. We consider the two-option opinion dynamics \eqref{eq:Nby2_shortened} with homogeneous parameters \eqref{eq: homogeneous params}, relaxing the  assumption of zero inputs: 
\begin{multline}
    \dot{x}_{i} = - d x_{i} + u \hat{S}_1 \left( {\alpha x_{i} + \gamma \textstyle\sum_{\substack{k\neq i\\k=1}}^\Na a_{ik} x_{k}}\right) + b_{i} \label{eq:Nby2_hom}.
\end{multline}

The next corollary follows from Corollary~\ref{cor:AgreeDisagree} and~\cite[Theorems~IV.1 and~IV.2]{BizyaevaCDC2021}. It recognizes the opinion-forming bifurcations of 
\eqref{eq:Nby2_hom} as agreement and disagreement pitchfork bifurcations and predicts their unfolding in response to distributed inputs as a function of network structure. In other words, it predicts the 
location of the two symmetric agreement (or disagreement) solutions  and how the input-driven unfolding selects one of the two solutions (see Figure~\ref{fig:bif}).

\begin{corollary} \label{cor:pitchfork}
Consider~\eqref{eq:Nby2_hom} and suppose that adjacency matrix $A$ is irreducible, i.e., the associated graph is strongly connected. Let $\lambda_{max}>0$ be the largest real-part eigenvalue of $A$, i.e. the Perron-Frobenius eigenvalue, with associated unitary positive right eigenvector $\mathbf{v}_{max}$ and unitary positive left eigenvector $\mathbf{w}_{max}$. Let $\lambda_{min}<0$ be the smallest real-part eigenvalue of $A$. Assume $\lambda_{min}$ is real, simple, and for all eigenvalues $\xi \neq \lambda_{min}$ of $A$, $\operatorname{Re}[\xi] \neq \lambda_{min}$. Let $\mathbf{v}_{min}$ and $\mathbf{w}_{min}$ be the right and left unitary eigenvectors associated to $\lambda_{min}$ with $\langle\mathbf{v}_{min},\mathbf{w}_{min}\rangle>0$.

\noindent A. \textbf{Cooperative agents.} If $\gamma>0$, inputs satisfy $\langle\mathbf{b},\mathbf{w}_{max}\rangle = 0$, 
and $\alpha+\lambda_{max}\gamma>0$, model~\eqref{eq:Nby2_hom} undergoes a supercritical pitchfork bifurcation for $u=u^{*} = \frac{d}{\alpha + \lambda_{max} \gamma }$ at which opinion-forming bifurcation branches emerge  from $\mathbf{x}=0$. The associated bifurcation branches are tangent at $\mathbf{x}=0$ to $\R\{\mathbf{v}_{max}\}$. 
The pitchfork unfolds in the direction given by $\langle\mathbf{b},\mathbf{w}_{max}\rangle$, i.e., if $\langle\mathbf{b},\mathbf{w}_{max}\rangle>0\,(<0)$, then the only stable equilibrium $\mathbf{x}^*$ for $u$ close to $u^*$ satisfies $\langle \mathbf{x}^*,\mathbf{v}_{max}\rangle>0\,(<0)$.

\noindent B. \textbf{Competitive agents.} If $\gamma<0$, inputs satisfy $\langle\mathbf{b},\mathbf{w}_{min}\rangle = 0$,
and $\alpha+\lambda_{min}\gamma>0$, model~\eqref{eq:Nby2_hom} undergoes a supercritical pitchfork bifurcation for $u=u^{*} = \frac{d}{\alpha + \lambda_{min} \gamma }$ at which opinion-forming bifurcation branches emerge from $\mathbf{x}=0$. The associated bifurcation branches are tangent at $\mathbf{x}=0$ to $\R\{\mathbf{v}_{min}\}$. 
The pitchfork unfolds in the direction given by $\langle\mathbf{b},\mathbf{w}_{min}\rangle$, i.e., if $\langle\mathbf{b},\mathbf{w}_{min}\rangle>0\,(<0)$, then the only stable equilibrium $\mathbf{x}^*$ for $u$ close to $u^*$ satisfies $\langle \mathbf{x}^*,\mathbf{v}_{min}\rangle>0\,(<0)$.
\label{propCDC}
\end{corollary}
\begin{remark}
    For \eqref{eq:Nby2} with homogeneous parameters \eqref{eq: homogeneous params} an analogous result to Corollary \eqref{cor:pitchfork} holds, except with $u^* = \frac{d}{\alpha - \beta + \lambda_{max/min} (\gamma - \delta)}$. 
\end{remark}

The symmetric opinion-forming pitchfork bifurcation predicted by Corollary~\ref{propCDC} in the case of trivial or balanced inputs $\langle\mathbf{b},\mathbf{w}_{max/min}\rangle = 0$ constitutes the simplest instance of multi-stability (bistability in this case) between different possible opinion states (see Figure~\ref{fig:bif} left for the disagreement case - an identical figure is found for the agreement case~\cite[Figure~1]{BizyaevaCDC2021}). For attention $u$ greater than the critical value $u^*$ (the bifurcation point), the group of agents can converge to either of the two stable opinion states depending on initial conditions and (in a real-world setting) unmodelled uncertainties and disturbances.

\begin{figure} 
    \centering
    \includegraphics[width=0.75\linewidth]{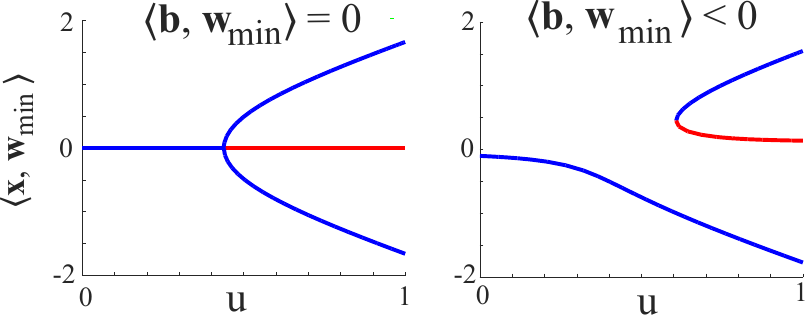}
\caption{\newtext{Bifurcation diagrams showing the symmetric pitchfork bifurcation (left) and its unfolding (right) for two-option opinion dynamics \eqref{eq:Nby2_hom} with $d=\alpha = 1$ in the disagreement regime ($\gamma =- 1$)  for three agents communicating over an undirected line graph. Blue (red) curves represent stable (unstable) equilibria. The vertical axis is  the projection of the system equilibria $\mathbf{x}$ onto $\mathbf{w}_{min}$ ($\mathbf{w}_{min}=\mathbf{v}_{min}$ since the graph is undirected). Left: $\mathbf{b} = (0.2,0,-0.2)$; right: $\mathbf{b} = -0.1 \mathbf{w}_{min} + (0.2,0,-0.2)$. Bifurcation diagrams  generated with help of MatCont 
    \cite{matcont}. In the agreement regime, the diagrams look qualititively the same with $\mathbf{w}_{min}$ replaced with $\mathbf{w}_{max}$ and $\mathbf{b}$ modified appropriately (see Figure 1 in \cite{BizyaevaCDC2021}). } }\label{fig:bif}
\end{figure}

In the agreement regime solutions on the upper  branch correspond to agents agreeing on option 1 and on the lower branch to agents agreeing on option 2. In the disagreement regime solutions on the upper branch correspond to one subgroup favoring option 1 and the second subgroup favoring option 2 and the lower branch to the first subgroup favoring option 2 and the second subgroup favoring option 1. Both the sign and relative magnitudes of the agent opinions are predicted by $\mathbf{v}_{max}$ in the agreement regime and $\mathbf{v}_{min}$ in the disagreement regime -- see Figure~\ref{fig:GraphClasses} for an illustration for four types of graphs. Observe that for the highly symmetric cycle graph, the group splits evenly in the disagreement case, whereas in the star and wheel graphs, the center node disagrees with all of the peripheral nodes. These results are easily predicted using well-known results on the eigenstructure of the adjacency matrix for these graphs. See~\cite{BizyaevaACC2021} for details. 

The symmetric pitchfork unfolds (Figure~\ref{fig:bif} right) in such a way that only one solution (that predicted by the sign $\langle\mathbf{b},\mathbf{w}_{max/min}\rangle$) is stable close to the symmetric bifurcation point. For larger values of the attention parameter, the other solution also regains stability in a saddle-node bifurcation but the input-driven asymmetry is still reflected in the size of the basin of attraction of the two solutions. The left eigenvectors of the adjacency matrix $\mathbf{w}_{max/min}$ define agreement/disagreement {\it centrality indices} because the unfolding formula $\langle\mathbf{b},\mathbf{w}_{max/min}\rangle\lessgtr 0$ implies that the larger $[\mathbf{w}_{max/min}]_i$ the larger the effect of a nonzero input $b_i$ on the agreement/disagreement pitchfork unfolding. Agreement and disagreement centrality indices can thus naturally be used to control opinion forming behavior via distributed inputs. By augmenting our opinion dynamics with an attention feedback mechanism, these centrality indices determines distributed thresholds for the triggering of opinion cascade, as illustrate in the next section (see also~\cite{franci2021analysis} for numerical illustrations on large random graphs and an application to task allocation in robot swarms). Finally, all the results in this section generalize to the case $\No>2$. This generalization requires the computation of the vector $v_{ax}$ appearing in Theorem~\ref{thm:Bifurcations} via equivariant bifurcation theory methods (see Remark~\ref{rmk:axial v}), a direction that we leave for future extensions of this work.}


\begin{figure}
    \centering
    \includegraphics[width = 0.47\textwidth]{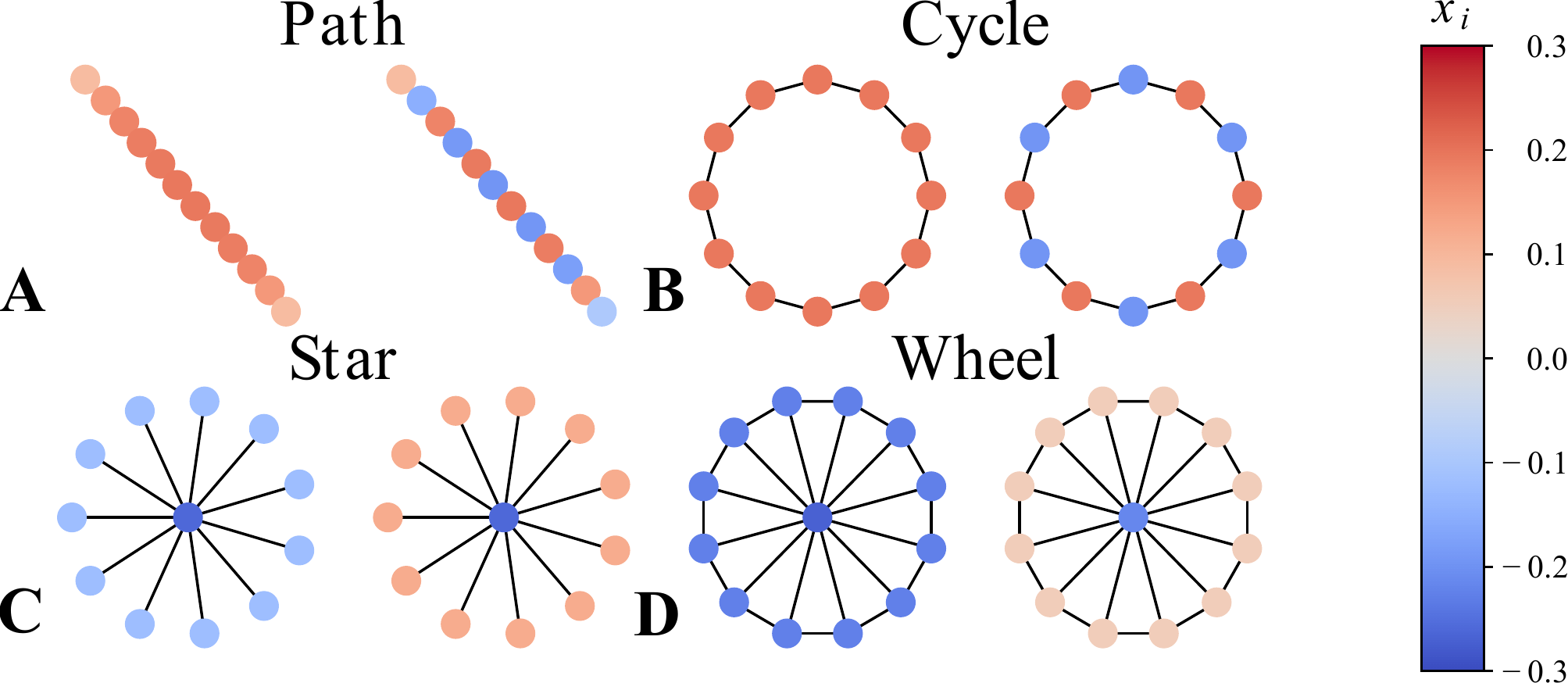}
    \caption{\newtext{Adapted from \cite{BizyaevaACC2021}. Agreement (left) and disagreement (right) opinion configurations at steady state  from simulation of two-option opinion dynamics \eqref{eq:Nby2_hom} and four different undirected graph types, with attention $u$ slightly above the critical value $u^*$ in each case. Color of each node $i$ corresponds to opinion $x_i$ at $t=500$. For all graphs  $\gamma = 1.3$  (left) and  $\gamma = -1.3$  (right), 
     $d = 1$, $\alpha = 1.2$, and $b_{i} = 0$ for all $i = 1, \dots, \Na$. For the path and cycle graphs, $u = 0.31$, and for the star and wheel graphs, $u = 0.26$.
    Randomized initial opinions are drawn from the distribution $U(-1,1)$.}}
    \label{fig:GraphClasses} 
\end{figure}

\subsection{Consensus and Dissensus Generic for \newtext{Transitive} Symmetry \label{sec:symmetry}}
\newtext{
In Section \ref{sec:bifurcations} we have shown how graph structure can inform what types of opinion configurations arise in the group. 
In this section we consider, for the homogeneous regime \eqref{eq: homogeneous params}, how 
the presence of symmetry in the communication graph can further constrains  opinion configurations. 
We show how consensus and dissensus emerge for  opinion dynamics \eqref{EQ:generic decision dynamics} 
with two different examples of transitive symmetry. 
We first introduce a few technical definitions from group theory and equivariant bifurcation theory. }

Let $\Ga$ be a compact Lie group acting on $\mathds{R}^{n}$.
Consider a dynamical system $ \dot{\mathbf{x}} = \mathbf{h}(\mathbf{x})$
where $\mathbf{x}\in \mathds{R}^{n}$ 
and $\mathbf{h}: \mathds{R}^{n} \to \mathds{R}^{n}$.  Then $\rho \in \Ga$ is a \textit{symmetry} of the system, equivalently $\mathbf{h}$ is $\rho$-\textit{equivariant}, if $\rho \mathbf{h}(\mathbf{x}) = \mathbf{h}(\rho \mathbf{x})$. 
If $\mathbf{h}$ is $\rho$-equivariant for all $\rho \in \Ga$, then $\mathbf{h}$ is $\Ga$-\textit{equivariant} \cite{GolubitskySymmetryPerspective}. 
$\Ga$-equivariance means elements of  symmetry group $\Ga$ send solutions to solutions.

The compact Lie group associated with permutation symmetries of $n$ objects is the \textit{symmetric group on $n$ symbols} $\ES_{n}$, which is the set of all bijections of  $\Omega_n := \{ 1, \ldots, n\}$ to itself (i.e., all permutations of ordered sets of $n$ elements). 
The opinion dynamics \eqref{EQ:generic decision dynamics} with homogeneous parameters \eqref{eq: homogeneous params} and  all-to-all coupling 
are {\em maximally symmetric},  i.e. $(\ES_{\No}\times \ES_{\Na})$-equivariant, where elements of $\ES_{\Na}$ permute the $\Na$-element set of agents and elements of $\ES_{\No}$ permute the $\No$-element set of options \cite{AF-MG-AB-NEL:20}. Maximally symmetric opinion dynamics are unchanged under any permutation of agents or  options.

A subgroup $\Ga_n\subset \ES_n$ is {\em transitive} if the orbit $\Ga_n(i)=\{\rho(i),\,\rho\in\Ga_n\}=\Omega$, for some (and thus all) $i\in\Omega$. $(\Ga_{\No}\times\Ga_{\Na})$-equivariant opinion dynamics, with transitive $\Ga_{\Na}$, are still {\em highly symmetric} since any pair of agents, while not necessarily interchangeable by arbitrary permutations, can be mapped into each other by the symmetry group action.  The following are examples of transitive subgroups of $\ES_{\Na}$:
\begin{itemize}
    \item $\D_{\Na}$,  dihedral group of order $\Na$; symmetries correspond to $\Na$ rotations and $\Na$ reflections. $\D_{\Na}$-equivariant opinion dynamics are unchanged if agents are permuted by a rotation or a reflection, e.g., if agents communicate over a network defined by an undirected cycle. 
    \item \newtext{$\Z_{\Na}$,  cyclic group of order $\Na$; symmetries correspond to $\Na$ rotations (and no reflections). $\Z_{\Na}$-equivariant opinion dynamics are unchanged if agents are permuted by a rotation, e.g., if agents communicate over a network defined by a directed cycle.}
\end{itemize}
Observe that the system opinion state space decomposes as $V = W_{c} \oplus W_{d}$, where  $W_c$ is the \textit{multi-option consensus space} defined as
\begin{equation}
W_c = \{(\Zz_1, \ldots, \Zz_{\Na}) \,| \, \Zz_i = \tilde{\Zz} \in \mathbf{1}_\No^\perp \, , \, \forall i \},
\label{eq:W1}
\end{equation}
and $W_{d}$ is the \textit{multi-option dissensus space} defined as
\begin{equation}
\label{eq:W2} W_d = \{(\Zz_1, \ldots,  \Zz_{\Na}) \,| \, \Zz_1 + \cdots +  \Zz_{\Na} = 0\}.
\end{equation}
On the consensus space $W_{c}$, agents have identical opinions. On the dissensus space $W_{d}$, agent opinions are balanced over the options such that the average opinion of the group is neutral.

Model-independent results~\cite[Theorem 4.6 and Remark 4.7]{AF-MG-AB-NEL:20} ensure that, in the presence of transitive symmetry, $\ker J=W_c$ or $\ker J=W_d$. I.e., if~(\ref{EQ:generic decision dynamics}) is symmetric with respect to a group $\Gamma_{a}$ that acts by swapping the agent indices transitively, then generically  $\ker J=W_c$ or $\ker J=W_d$. In the homogeneous regime 
\eqref{eq: homogeneous params},  agent symmetry of 
(\ref{EQ:generic decision dynamics}) is fully determined by  
$A$ as proved in the following proposition for the maximally symmetric case $\Ga_{a}=\ES_{\Na}$ and the highly symmetric case $\Ga_{a}=\D_{\Na}$ (see Appendix~\ref{app:transSym} for proof). The same result holds, with similar proof, for other transitive agent symmetries, e.g., $\Ga_a=\Z_{\Na}$.

\begin{proposition}\label{PROP:realization symmetries}
Consider model~(\ref{EQ:generic decision dynamics}) in the homogeneous regime defined by~\eqref{eq: homogeneous params}. Then the following hold true:
\begin{enumerate}
    \item Model~(\ref{EQ:generic decision dynamics}) is $(\ES_{\No}\times \ES_{\Na})$-equivariant if and only if $A$ is the adjacency matrix of an all-to-all graph;
    \item If $A$ is the adjacency matrix of an undirected cycle graph, then model~(\ref{EQ:generic decision dynamics}) is $(\ES_{\No}\times \D_{\Na})$-equivariant.
\end{enumerate}

\end{proposition}

\begin{remark}
    More generally, the symmetry group of the opinion dynamics is determined by the automorphism group of the graph associated to $A$. The proof follows
    as for Proposition~\ref{PROP:realization symmetries}. \label{remark:symmetry}
\end{remark}


The next corollary follows from Theorem~\ref{thm:Bifurcations} and~\cite[Theorem 4.6 and Remark 4.7]{AF-MG-AB-NEL:20}. The 
two types of opinion-formation behaviors proved in this corollary, i.e., consensus for cooperative agents and dissensus for competitive agents, respectively, constitute an opinion-formation analogue of consensus and balanced (split) states in coupled phase oscillators (see, e.g., \cite{OlfatiSaber2004,sepulchre2007stabilization,sepulchre2008stabilization,strogatz2000kuramoto}).

\begin{corollary}[
Consensus from Cooperation and Dissensus from Competition]
	Consider model~(\ref{EQ:generic decision dynamics}) in the homogeneous regime~\eqref{eq: homogeneous params}. Suppose that the graph associated to adjacency matrix $A$ is either all-to-all or an undirected cycle. Let $u_a$ and $u_d$ be defined by \eqref{eq:uc} and \eqref{eq:ud}. 
	
	\noindent A. {\bf Cooperative agents and consensus}. If agents are cooperative ($\gamma-\delta>0$), then opinion formation appears as a bifurcation along the consensus space at $u=u_a$ with $\lambda = \Na-1$ for the all-to-all case and $\lambda = 2$ for the cycle case.
	
	\noindent B. {\bf Competitive agents and dissensus}. If agents are competitive ($\gamma-\delta<0$), then opinion formation appears as a bifurcation along the dissensus space at $u=u_d$ with $\lambda=  -1$ for the all-to-all case, $\lambda = -2$ for the cycle case, when $\Na$ is even, and $\lambda= 2 \cos(\pi(\Na - 1)/\Na)$, when $\Na$ is odd.
 \label{cor:CD_conditions}
\end{corollary}

\begin{figure}
    \centering
    \includegraphics[width=0.9\columnwidth]{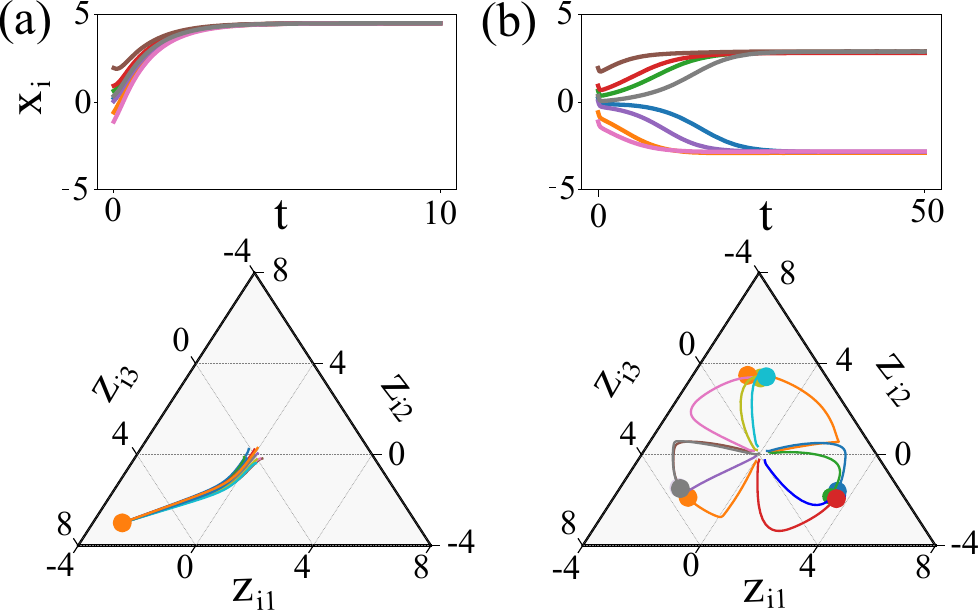}
    \caption{Simulations for $\No = 2$ options and $\Na = 8$ agents (top) and $\No = 3$ options and $\Na = 12$ agents (bottom). Opinions form (a) consensus when agents are cooperative:
    $\gamma = 0.2$, $\delta = -0.1$; (b) dissensus when agents are competitive: $\gamma = -0.1$, $\delta = 0.2$. In each plot, $\alpha = 0.2$, $\beta = 0.1$, $d = 1$, $u = 3$, $\hat{b} = 0$, and random initial conditions are the same. Communication weights $\alpha, \beta, \gamma, \delta$ were perturbed with small random additive perturbations drawn from a normal distribution with variance (a) 0.01, (b) 0.001. Ternary plots for three options generated with the help of \cite{pythonternary}.  
    }
    \label{Fig:ConesnsusDissensusTrajectories} 
\end{figure}

\begin{remark}[Stability of Consensus and Dissensus]
Consensus and dissensus solution branches predicted for the symmetric networks in Corollary \ref{cor:CD_conditions} are a consequence of the Equivariant Branching Lemma \cite[Section 1.4]{GolubitskySymmetryPerspective}, and are made of hyperbolic equilibria. Their stability can be proved using the tools in \cite[Section XIII.4]{Golubitsky1988} and \cite[Section 2.3]{GolubitskySymmetryPerspective}.
\end{remark}

\section{Attention dynamics and tunable sensitivity \label{sec:feedback}}

We have established that existence of agreement and disagreement equilibria and multistability of opinion formation outcomes arise from bifurcations of the general opinion dynamic model (\ref{EQ:generic decision dynamics}). In this section we explore how {\em ultra-sensitivity to inputs $b_{ij}$, robustness to changes in inputs, and opinion cascade dynamics} also arise from bifurcations. With the addition of dynamic state feedback for model parameters in (\ref{EQ:generic decision dynamics}), the opinion formation process can reliably amplify arbitrarily small inputs $b_{ij}$, reject small changes in input as unwanted disturbance, facilitate an opinion cascade even if only one agent gets an input, and enable groups to move easily between consensus and dissensus.  The choice of feedback design parameters determine implicit thresholds that make all of these behaviors tunable.

\newtext{The addition of dynamic state feedback for parameters in our model is similar in spirit to the extension of linear weighted-average model with nonlinear state-feedback update rules for the coupling gains, as in bounded confidence models \cite{Deffuant2000,HegselmannKrause2002,hegselmann_opinion_2005}  and biased assimilation models \cite{Dandekar2013,xia2020analysis}. However, our motivation, rather than to capture a specific sociological phenomenon, is to make our model 
adaptable to  inputs and flexibly responsive to changing environments. This is achieved by ensuring tunable sensitivity of opinion formation to  inputs. We illustrate our ideas and prove our results for the case $\No=2$. The multi-option extension is left for future work.}



\subsection{Dynamic State Feedback Law for Attention \label{sec:feebdacklaw}}
In the same spirit as~\cite{zhong2019continuous,Gray2018}, we augment the opinion dynamics (\ref{EQ:generic decision dynamics}) by introducing feedback dynamics on the attention parameter $u_{i}$ for each agent $i$, in the form of a leaky integrator with saturared input
\begin{align}
    \tau_{u} \dot{u}_{i} & = - u_{i} + S_{u}\left( \frac{1}{N_o}\textstyle\sum_{k = 1}^{\Na} \textstyle\sum_{l = 1}^{\No} \left(\bar{a}_{ik} z_{kl}\right)^{2} \right)\,,
    \label{eq:u_feedback}
\end{align}
a simple dynamics universally found, in particular, in decision making models~\cite{MusslickBizyaevaCogSci,UM2001,Bogacz2007}.
Here, $\tau_{u}>0$ is a time scale, which can be freely chosen. $S_{u}$ is a smooth saturating function, satisfying $S_{u}(0) = 0$, $S_{u}(y) \to \bar u >0$ as $y \to \infty$, $S_{u}'(y) > 0$ for all $y \in \mathds{R}$, and $S'''_{u}(y) > 0$ for all $y > 0$.  We define $S_u$ as a Hill saturating function
\begin{equation}
    S_{u}(y) = \underline u + (\bar u - \underline u) \frac{y^{n}}{(y_{th})^{n} + y^{n}} \label{eq:Su}, 
\end{equation}
where threshold $y_{th} > 0$ and $n>0$.  In \eqref{eq:Su} we constrain $\bar u$ and $\underline u$ such that  $\bar u > u_c \geq \underline u > 0$, with $u_c = u_a$ ($u_d$) when $\gamma>0$ ($<0$) and $u_a$, $u_d$ are defined by \eqref{eq:uc},\eqref{eq:ud}. For the remainder of this section we consider the homogeneous regime \eqref{eq: homogeneous params}, except for the $u_i$, which are heterogeneous. The attention coupling matrix $\bar A$ with elements $\bar a_{ik}$ can be distinct from the opinion coupling matrix $A$ but here we let $\bar{A} = A + \mathcal{I}_{\Na}$.  
For $N_o = 2$ the attention feedback dynamics \eqref{eq:u_feedback} simplify to
\begin{equation}\label{eq:u_feedback_2op}
    \tau_{u} \dot{u}_{i} = - u_{i} + S_{u}\left( \textstyle\sum_{k = 1}^{N_{a}} (\bar{a}_{ik}x_{k})^{2} \right).
\end{equation}

\subsection{Tunable Sensitivity and Robustness for a Single Agent}
\label{sec:single agent}

\newtext{In this section we first consider a single agent with dynamic opinions \eqref{eq:Nby2} and dynamic attention~\eqref{eq:u_feedback_2op} with no neighbors, i.e., $a_{ik} = 0$ for all $k = 1, \dots, N_a$. As shown in Figures~\ref{fig:tunable_sensitivity} and~\ref{fig:tunable_robustness}, the equilibria of the coupled opinion and attention dynamics can graphically be represented as the intersection of the {\em $x_i$-nullcline} $\{\dot x_i=0\}$ (black solid) and {\em $u_i$-nullcline} $\{\dot u_i=0\}$ (red dashed) in the $(u_i,x_i)$ plane. Corollary \ref{cor:pitchfork}  defines the shape of the $x_i$-nullcline as a pitchfork bifurcation which unfolds with nonzero input $b_i$, analogous to Figure~\ref{fig:bif}. }

For model \eqref{eq:Nby2},\eqref{eq:u_feedback_2op}, define agent $i$ to be \textit{strongly opinionated} when its attention is close to its upper saturation value, i.e., $u_{i} \simeq \bar u$, and \textit{weakly opinionated} when its attention is close to its lower saturation value, i.e., $u_{i} \simeq \underline u$. What we refer to as {\em tunable sensitivity of opinion formation to input $b_i$} can then be understood by comparing the plots of Figure~\ref{fig:tunable_sensitivity}, where the opinion trajectory for agent $i$ is plotted on the left for $b_i=0.5$ and on the right for $b_i=1$. For the given  parameters and $b_i=0.5$, the nullclines intersect at three points in the positive half-plane. For unopinionated initial conditions, the opinion state is attracted to the point corresponding to a weakly opinionated equilibrium: agent $i$ {\em rejects the input $b_i=0.5$} and does not form a strong opinion.  For the same parameters and  $b_i=1$, the nullclines intersect at only one point, corresponding to a strongly opinionated equilibrium.  Thus, for the same initial conditions, agent $i$ {\em accepts the input $b_i=1$} and forms a strong opinion.  The implicit {\em sensitivity threshold}\footnote{\newtext{Quantifying the exact relationship between the design parameters in the saturation function \eqref{eq:Su} and the implicit thresholds described in this section is a straightforward but lengthy calculation, which involves taking implicit derivatives of the equilibria of the coupled system with respect to the design parameters. Due to space constraints we leave out this analysis here.}} that distinguishes  rejected from  accepted inputs can be tuned by using parameters $n, y_{th}$ in \eqref{eq:Su}. Changing their value changes the shape of the $u_i$-nullcline and thereby varies how strong of an input $b_i$ is required to reduce the number of nullcline intersections from three to one, as in Figure~\ref{fig:tunable_sensitivity}.

{\em Tunable robustness of opinion formation to changes in input $b_i$} can be understood by comparing the sequence of plots in the top and bottom halves of Figure~\ref{fig:tunable_robustness}. The plots on the left show agent $i$ forming a strong opinion in the direction of the input $b_i=1$.  The plots on the right show what happens to agent $i$'s opinion when the input switches to $b_i =-1$, i.e., an input that is in opposition to the original input.  In the top sequence, when $\bar u =1$,  agent $i$ {\em accepts the change of input} and forms a strong opinion in the direction of the new input.  In the bottom sequence, when $\bar u=2.5$, agent $i$ {\em rejects the change of input} and retains a strong opinion in the direction of the original input.  The implicit {\em robustness threshold} that distinguishes rejected from accepted changes in input can be tuned by design parameter $\bar u$.

\begin{figure}
    \centering
    \includegraphics[width=0.4\textwidth]{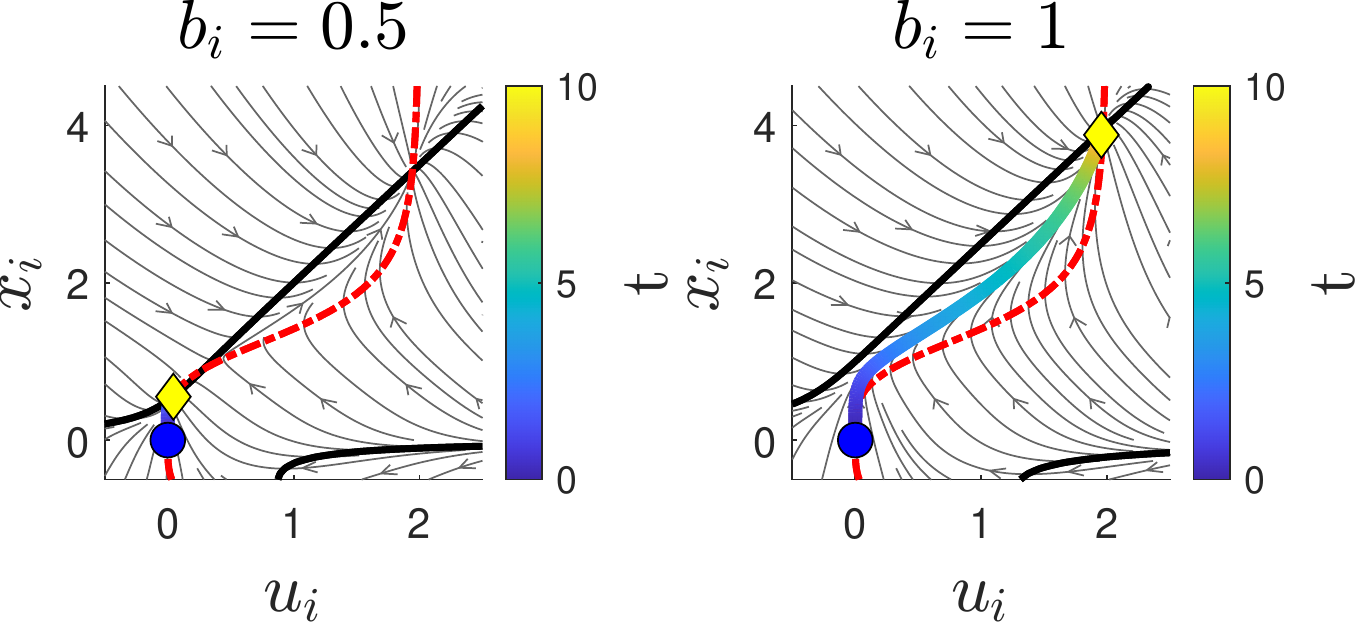}  
    \caption{Sensitivity of opinion formation to input magnitude. $(u_i,x_i)$-phase plane and trajectories of  (\ref{eq:Nby2}),(\ref{eq:u_feedback_2op}); 
    $n=2,y_{th} = 0.4,\alpha_i = 2, \beta_i = -1, \gamma_{ik} = \delta_{ik}=0,  d_i = 1, \tau_{u} = 1,\underline u = 0, \bar u = 2$ for $b_i=0.5$ (left) and $b_i = 1$ (right). 
    Initial state $(u_{i}(0),x_{i}(0)) = (0,0)$ is a blue circle, and final state a yellow diamond.  Nullclines of (\ref{eq:Nby2}) are black solid 
    and (\ref{eq:u_feedback_2op}) are red dashed.
    Gray arrows show flow streamlines. Color scale is time.}
    \label{fig:tunable_sensitivity}
\end{figure}

\begin{figure}
    \centering
    \includegraphics[width=0.4\textwidth]{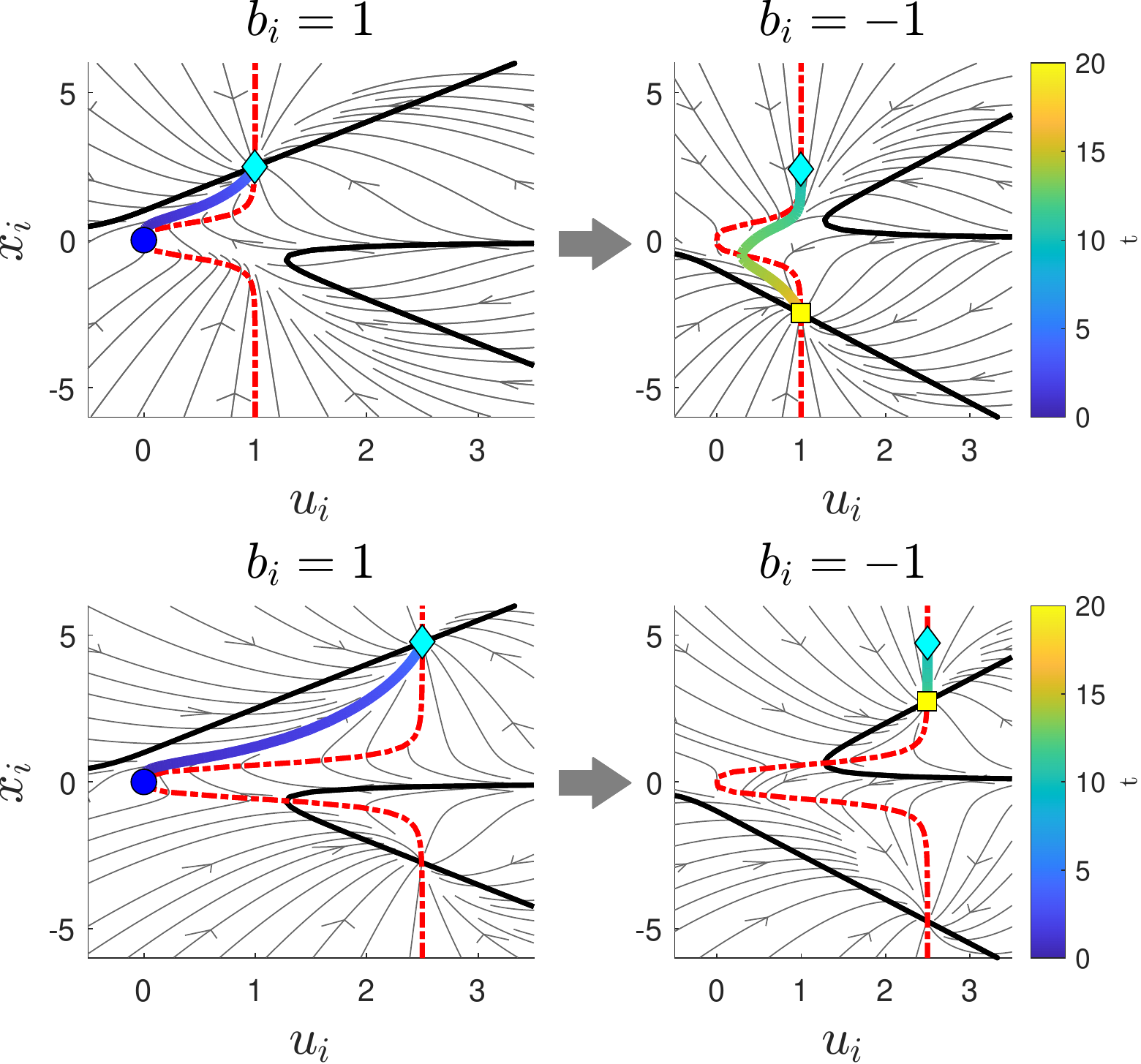}
    \caption{Robustness of opinion formation to changes in input. $(u_i,x_i)$-phase plane and trajectories of (\ref{eq:Nby2}),(\ref{eq:u_feedback_2op});
    $n=2, y_{th} = 0.4, \alpha_i = 2, \beta_i = -1, \gamma_{ik} = \delta_{ik} = 0, d_i = 1, \tau_{u} = 1, \underline u = 0$. (Left) Input is $b_i = 1$, initial state $(u_{i}(0),x_{i}(0)) = (0, 0)$ is a blue circle, and final state is a cyan diamond.  (Right) Input changes to $b_i=-1$, initial state is final state on left and final state is yellow square. Top: $\bar u = 1$, and agent changes opinion in direction of new input. Bottom: $\bar u = 2.5$, and agent retains opinion in original direction.   Nullclines, streamlines, and time are drawn as in Figure~\ref{fig:tunable_sensitivity}. }
    \label{fig:tunable_robustness}
    
    
\end{figure}

\newtext{
\subsection{Opinion Cascades with Tunable Distributed Sensitivity \label{sec: opinion cascades}}

The following corollary shows that our feedback attention dynamics create a distributed threshold for the opinion dynamics below which the agents remain weakly opinionated and above which agents converge to a strongly opinionated equilibrium. The transition from a weakly opinionated to a strongly opinionated equilibrium in response to inputs is called an {\em opinion cascade}. The threshold is defined in terms of the inner product of the vector of inputs $\mathbf{b}$ and suitable eigenvectors of the opinion network adjacency matrix. In other words, the threshold is distributed across the agents and the spectral properties of the adjacency matrix determine highly sensitive and weakly sensitive directions in the input vector space. As in Section~\ref{sec:single agent} for single agents, the threshold can be tuned with  parameters of the attention dynamics.

In the following theorem, we let $\lambda_{max}$, $\mathbf{w}_{max}$ and $\lambda_{min}$, $\mathbf{w}_{min}$ 
satisfy the  assumptions  of Corollary~\ref{propCDC}.\footnote{\newtext{The proof of Theorem~\ref{thm:cascade} follows from~\cite[Theorem~V.3]{BizyaevaCDC2021} and from geometric arguments based on implicit differentiation, similarly to the single-agent case of Section~\ref{sec:single agent}. It is omitted for space constraints.}}

\begin{theorem}\label{thm:cascade}
Consider the coupled system \eqref{eq:Nby2_shortened},\eqref{eq:u_feedback} with $d_i=d$ $\alpha_i=\alpha$, and $\gamma_{ik}=\gamma a_{ik}$, where $A=[a_{ik}]$ is a symmetric and irreducible adjacency matrix. Let $u_{c} = \frac{d}{\alpha+\lambda_{max}\gamma}$, $\mathbf{w}_c = \mathbf{w}_{max}$ if $\gamma > 0$ and $u_{c} = \frac{d}{\alpha+\lambda_{min}\gamma}$, $\mathbf{w}_c = \mathbf{w}_{min}$ if $\gamma < 0$. There exists $\varepsilon>0$ such that for $0<u_c-\underline u,y_{th}<\varepsilon$ and $n$ sufficiently large, the following generically  hold. There exists $p = p(y_{th})>0$ satisfying $\frac{\partial p}{\partial y_{th}} > 0$ such that, for $|\langle \mathbf{w}_c, \mathbf{b} \rangle | < p$, model~\eqref{eq:Nby2},\eqref{eq:u_feedback} possesses a weakly opinionated locally exponentially stable equilibrium. This equilibrium loses stability in a saddle-node bifurcation for $|\langle \mathbf{w}_c, \mathbf{b} \rangle | = p$. No weakly opinionated equilibria exist for $|\langle \mathbf{w}_c, \mathbf{b} \rangle | > p$ and all trajectories converge to a strongly opinionated agreement (disagreement) equilibrium for $\gamma>0$ ($\gamma < 0$). 
For $\gamma=0$, the strongly opinionated equilibrium $(\mathbf{x}^*,\mathbf{u}^*)$ satisfies $\operatorname{sign}(x_i^*)=\operatorname{sign}(b_i)$.
\end{theorem}

\begin{figure}
    \centering
    \includegraphics[width=0.85\columnwidth]{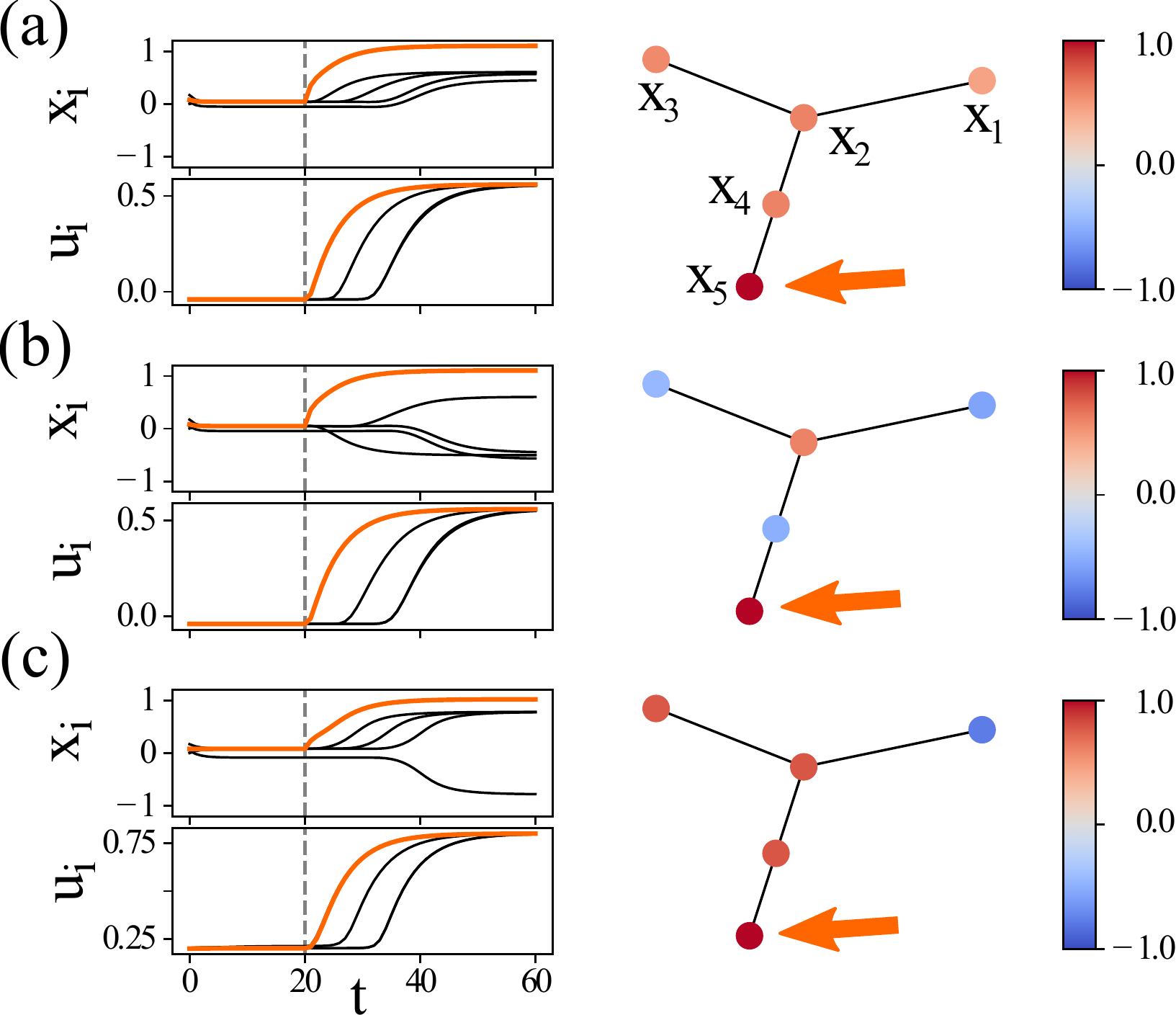}
    \caption{\newtext{Opinion cascades with opinion and attention dynamics defined in Theorem~\ref{thm:cascade}. 
    For $t < 20$, 
    $\mathbf{b} = (-0.05,0.05,0.05,0.05,0.05)$ for all three simulations. At $t = 20$ the input to agent 5 (indicated by the arrow) increases to $b_5 = 0.25$, which triggers an opinion cascade on the network. Plots show opinion and attention trajectories of the agents with agent 5 in orange. Network diagrams on the right show the opinion strength of each agent at $t = 60$ of the simulation. (a) Agreement cascade with $\gamma = 1$, the network chooses the positive option following the informed agent. (b) Disagreement cascade with $\gamma = -1$, agents' opinions on the network disperse following the sign structure of $\mathbf{v}_{min}$. (c) Agents are coupled through the attention dynamics only (i.e. $\gamma = 0$), opinion cascade causes each agent to amplify their small input and commit to a strong opinion. Other parameters: $\alpha = 2$, $n = 3$, $y_{th} = 0.1$, $\tau_u = 5$, $d = 1$, $\bar u = u^* + 0.3$, $\underline u = u^* - 0.3$, $u_i(0) = \underline u$ for all $i = 1, \dots, N_a$. $\mathbf{x}(0)$ generated randomly from a uniform distribution between $-0.2$ and $0.2$; the same initial condition was used for all three simulations.} 
    }
    \label{fig:three_cascades}
\end{figure}

Figure~\ref{fig:three_cascades} 
illustrates the predictions of Theorem~\ref{thm:cascade}. It shows that the arrival of a suprathreshold input at $t=20$ triggers an opinion cascade. Independently of the entries of the input vector $\mathbf{b}$, the cascade goes to a strongly opinionated agreement equilibrium for $\gamma>0$ (Figure~\ref{fig:three_cascades}a) and to a strongly opinionated disagreement equilibrium for $\gamma<0$ (Figure~\ref{fig:three_cascades}b). Conversely, for $\gamma=0$, the pattern of opinions at the strongly opinionated equilibrium is determined by the sign of the entries of the input vector. Figure~\ref{fig:cascade_sims} makes these observations more quantitative by showing the cascade threshold predicted by Theorem~\ref{thm:cascade} as a joint function of the norm of the input vector and of the cosine of the angle between the input vector and the relevant eigenvector of the adjacency matrix. As predicted by the theorem, when the input vector is misaligned with respect to the adjacency matrix eigenvector, large-magnitude inputs are necessary to robustly trigger an opinion cascade. Conversely, as the two vectors align, an opinion cascade can be triggered with much smaller inputs.
}

\begin{figure}
    \centering
    \includegraphics[width=0.9\columnwidth]{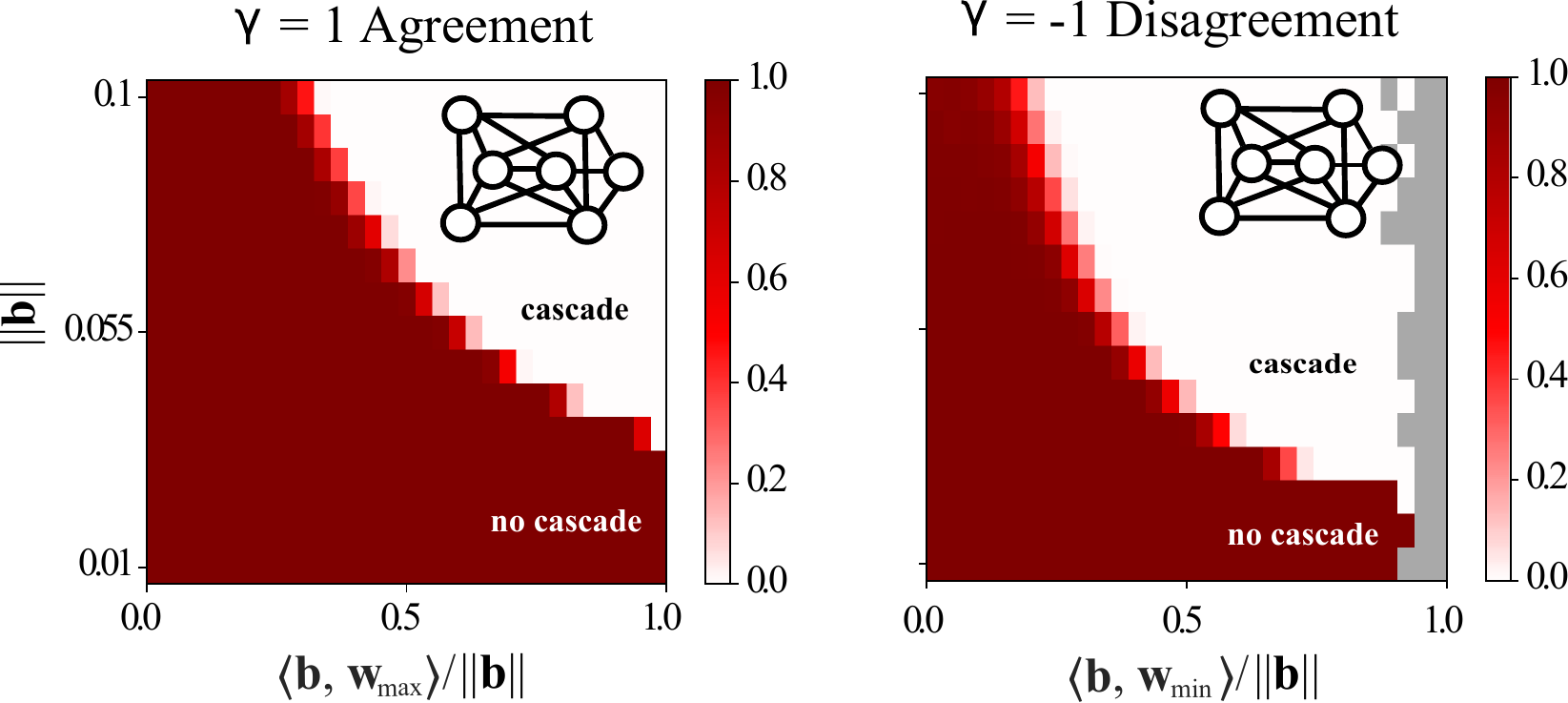}
    \caption{\newtext{Adapted from \cite{BizyaevaCDC2021}. Frequency of agreement (left) and disagreement (right) cascades for opinion and attention dynamics defined in Theorem~\ref{thm:cascade}. Color represents proportion of simulations in the given parameter range that did not result in a network cascade by $t=500$. Dark red corresponds to no cascades, white  to  always  a cascade, and grey to bins with no datapoints. Each plot shows the results of $1.5\times 10^5$ distinct simulations with $\tau_u = 10$, $y_{th} = 0.2$, $\underline u = u_{a} - 0.01$ for $\gamma = 1$ (left) and  $\underline u = u_{d} - 0.01$ for $\gamma = -1$ (right). For every simulation, initial conditions were $x_i = 0$, $u_i = 0$ for all $i = 1, \dots, \Na$ and $b_i$ were drawn from $\mathcal{N}(0,1)$ with  $\mathbf{b}$  normalized to a desired constant magnitude. 10000 simulations were performed for each constant input magnitude, with 15 magnitudes sampled uniformly spaced between $0$ and $0.1$. 
    } }
    \label{fig:cascade_sims}
\end{figure}

\section{Agreement -- disagreement transitions \label{sec:transition}}
We  illustrate 
how  feedback dynamics of social influence weights in the two-option opinion dynamics \eqref{eq:Nby2} can be used to facilitate transitions between agreement and disagreement on the network. 
Suppose agents comprise two clusters of size $N_1$ and $N_2$ with  index sets $\mathcal I_1$  and $\mathcal I_2$. 
Let $b_i= b_p$ for $i \in \mathcal I_p$ and $\hat x_p = \frac{1}{N_{p}} \sum_{i \in \mathcal{I}_{p}} x_{i}$, where $p \in \{ 1, 2 \}$.   
We define intra-cluster coupling as $\gamma_{ik}=\alpha/N_p>0$ and $\delta_{ik}=\beta/N_p<0$, $l\neq j$, $p=1,2$, $d_i = d$ for all $i,k \in \mathcal{I}_{p}$, and agent attention dynamics  by~\eqref{eq:u_feedback_2op} with $\bar a_{ik}=1$ for all $i,k$.

The influence network between the clusters is dynamic. We define feedback dynamics for the inter-cluster coupling as
\begin{subequations}
\begin{gather}
    \tau_{\gamma} \dot{\gamma}_{i} = - \gamma_{i} + \sigma  S_{\gamma}(\hat{x}_{1}\hat{x}_{2}) \\
    \tau_{\delta} \dot{\delta}_{i} = - \delta_{i} - \sigma S_{\delta}(\hat{x}_{1}\hat{x}_{2})
\end{gather}\label{eq: gamma delta dynamics}\end{subequations}
\noindent where $\sigma \in \{ 1, -1 \}$, $\tau_{\gamma}, \tau_{\delta} >0$ are time scales, $S_{\gamma}(y)=\gamma_{f} \tanh(g_{\gamma}y)$, $S_{\delta}(y)=\delta_{f} \tanh(g_{\delta}y)$, and $\gamma_f, \delta_f, g_{\gamma}, g_{\delta} >0$.


The sign of design parameter $\sigma$ in \eqref{eq: gamma delta dynamics} determines whether the system tends towards consensus or dissensus, and switching the sign can reliably trigger a transition between agreement and disagreement. 
Figure \ref{fig:control2clusters} illustrates the opinion formation of 7 agents that form two clusters, one with 3 agents and the other with 4 agents. One cluster has input favoring option 1 and the second favoring option 2. Initially, $\gamma - \delta < 0$ on average and the clusters evolve to a dissensus state which is informed by the agents' inputs. However, because  $\sigma = 1$, the two clusters eventually evolve towards a consensus state once $\gamma-\delta >0$ despite the inputs favoring disagreement. At time $t=300$, 
$\sigma$ switches sign to $\sigma = -1$ and the two clusters evolve back towards a clustered dissensus state once $\gamma-\delta < 0$. 

 \begin{figure}
        \centering
        \begin{tikzonimage}[trim=0 0 0 0, clip=true, width=0.8\columnwidth]{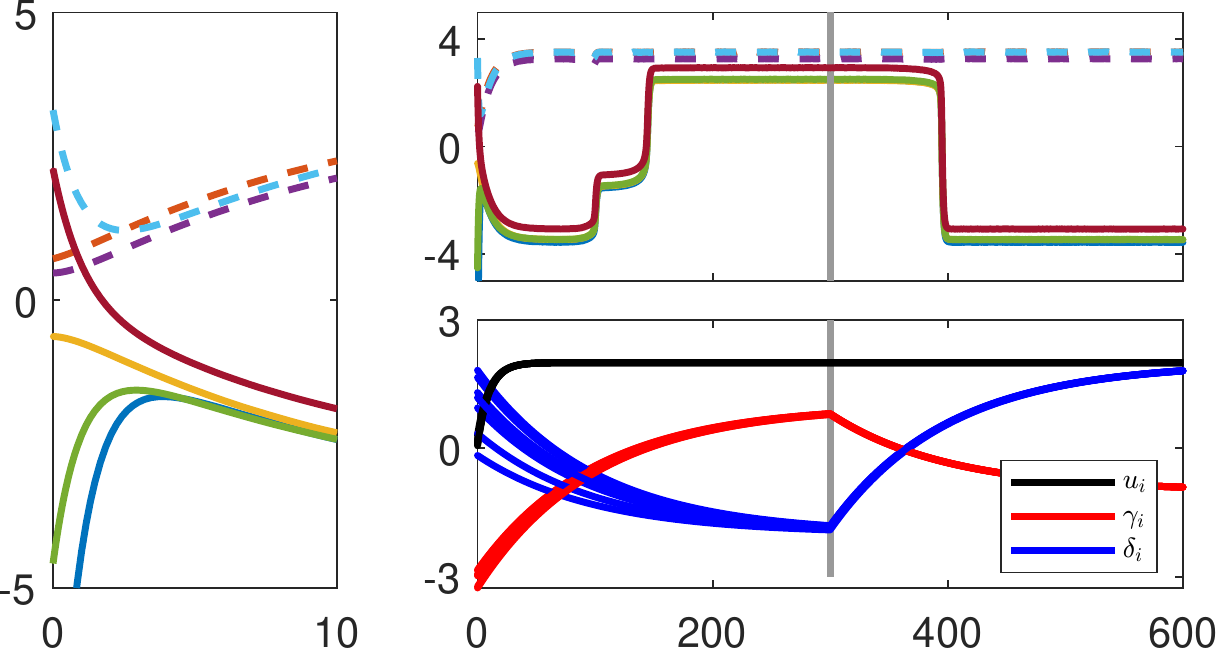}
        \node[rotate=0] at (-0.05,0.97) {(a)};
        \node[rotate=0] at (0.32,0.97) {(b)};
        \node[rotate=0] at (0.17,0.03) {\large$t$};
        \node[rotate=0] at (0.67,0.03) {\large$t$};
        \node[rotate=90] at (-0.03,0.55) {\large $x_{i}$};
        \node[rotate=90] at (0.33,0.77) {\large $x_{i}$};
        \end{tikzonimage}
        \caption{(a) Transient opinion trajectories settling to the clustered attractive manifold from random initial conditions in a simulation; (b) Full simulation. Top: opinion trajectories; Bottom: parameter trajectories. Seven agents form two clusters of sizes $N_{1} = 3$ (dashed-line opinion trajectories), $N_{2} = 4$ (solid-line opinion trajectories).
        Parameters are $d = 1$, $\alpha = 1$, $\beta = -1$, $b_{1} = 0.5$, $b_{2} = -0.5$  $\tau_{u} = 10$, $\tau_{\gamma} = \tau_{\delta} = 100$, $\gamma_{f} =2$, $\delta_{f} = 1$, $\underline u = 2$, $ g_{\gamma} = g_{\delta} = 10$, $y_{m} = 1.5$. Initial conditions $x_{i}(0)$ are drawn from $\mathcal{N}(0,2)$, $u_{i}(0)$  from $\mathcal{N}(0,0.3)$, $\gamma_{i}(0)$ from $\mathcal{N}(-3,0.3)$, and $\delta_{i}(0)$ from $\mathcal{N}(1,0.3)$. Also,  parameters $d_i$, $\alpha_i$, $\beta_i$, $b_{i}$ have additive perturbations drawn from $\mathcal{N}(0,0.1)$ independently for each agent $i$. For $t < 300$, $\sigma = 1$ and for $t \geq 300$, $\sigma = -1$.} 
        \label{fig:control2clusters}
        \vspace*{-5mm}
    \end{figure}
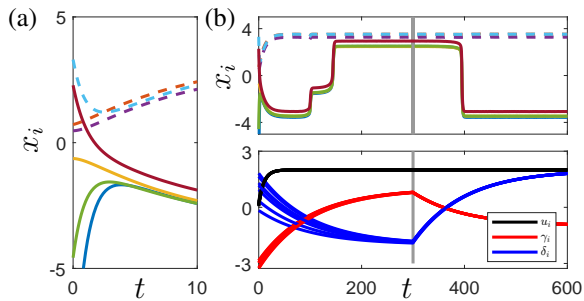

\section{Final Remarks}
\label{sec: final remarks}

\newtext{
Our opinion dynamics provide a new modeling framework for studying a variety of phenomena in which opinion formation is the governing behavior. In contrast to previous models, our approach focuses on the intrinsic nonlinear nature of opinion exchanges and thus on bifurcations as the key mechanism for analyzing and controlling opinion formation. Our model exhibits the flexibility, adaptability and robustness of natural opinion-forming systems, including deadlock-breaking and tunable sensitivity to changing inputs. A special instance of our model was  motivated by modeling decision making in honeybee communities~\cite{Gray2018}. The analytical tractability of our model makes it possible to tackle its rich dynamical behavior constructively. This has allowed us to make novel predictions about the role of the opinion network structure in determining the emerging patterns of opinion formations and the sensitivity of the network to exogenous inputs, as well as to design adaptive feedback control laws for the model parameters.

The applicability of our model to real-world problems has recently been confirmed by our recent contributions in sociopolitical problems~\cite{Leonard2021}, the design of task-allocation algorithms in robot swarms~\cite{franci2021analysis}, cognitive control~\cite{MusslickBizyaevaCogSci}, and game theory~\cite{ParkAllerton2021}. Other possible applications include decision making in biological and artificial neural networks, epidemiology and disease spread, and decision making in groups, from humans and robots to bacteria and animals on the move. }

\newtext{
\section*{Acknowledgements}
We acknowledge the contributions of Ayanna Matthews and Timothy Sorochkin to Figures~\ref{fig:GraphClasses} and~\ref{fig:cascade_sims}.
}

\appendix

\subsection{Extension to Heterogeneous Inter-option Coupling}
\label{app:adjacencytensor}

In future applications of the opinion dynamics model \eqref{EQ:generic decision dynamics} it may be useful to consider scenarios in which there is a heterogeneous level of influence between different options, i.e., in addition to the inter-agent interaction network there is an {\em inter-option interaction network}. In order to capture this, we introduce the \textit{adjacency tensor} whose entries $A_{ik}^{jl}$ capture the weight of influence agent $k$'s opinion on option $l$ has on agent $i$'s opinion on option $j$, which leads to the generalized opinion dynamics:
\begin{subequations}
	\begin{align*}
	& \quad \; \; \;\dot \Zz_i=P_0{\bs F}_i(\Zz)\\
	&F_{ij}(\Zz)=-d_{i}z_{ij} +  
	u_{i}\textstyle\sum_{l = 1}^\No S_{l}\left( \textstyle\sum_{k =1}^\Na A_{ik}^{jl} z_{kl}\right) + b_{ij}.
	\end{align*}
\end{subequations}
The model studied in this paper is recovered when 
$S_l$ is $S_1$ for same-option interactions and $S_2$ for inter-option interactions, and 
$A_{ii}^{jj} = \alpha_i$, $A_{ik}^{jj} = \gamma_{ik}$, $A_{ii}^{jl} = \beta_i$, and $A_{ik}^{jl} = \delta_{ik}$ for all $i,k = 1, \dots, \Na$, $j,l = 1, \dots, \No$, $i \neq k$,$j \neq l$.

\subsection{Well-definedness of Model \label{app:simplex}}
We show that the general model  \eqref{EQ:generic decision dynamics} is well defined by showing in Lemma~\ref{lem:forwardinv} that $V$ is forward invariant for \eqref{EQ:generic decision dynamics} and in Theorem~\ref{THM:realization boundedness} that solutions  are bounded.  
\newtext{ We define $D = \operatorname{diag}\{d_1, \dots, d_{\Na} \} \otimes \mathcal{I}_{\No}$. }

\begin{lemma} \label{lem:forwardinv}
$V$ is forward invariant for~\eqref{EQ:generic decision dynamics}.
\end{lemma}
\begin{proof}
For all $i=1,\ldots,\Na$, $\sum_{j=1}^\No \dot z_{ij}=0$, so if $z_{i1}(0) +  \cdots + z_{i\No}(0) = 0$, $z_{i1}(t) +  \cdots + z_{i \No}(t) = 0$ for all $t>0$. 
\end{proof}

\begin{theorem}[Boundedness]\label{THM:realization boundedness}
Let $\bar U$ be a compact subset of $\R$. There exists $R>0$ such that, for all $u_{i},d_{i},\alpha_{i}, \beta_{i}, \gamma_{ik}, \delta_{ik},b_{ij}\in\bar U$, $i,k=1,\ldots,\Na$, $j,l=1,\ldots,\No$, the set $V\cap\{|z_{ij}|\leq R,\,i=1,\ldots,\Na,\,j=1,\ldots,\No\}$ is forward invariant for~\eqref{EQ:generic decision dynamics}. This implies that the solutions $\Zz(t)$ of the dynamics \eqref{EQ:generic decision dynamics} are bounded for all time $t \geq 0$. 
\end{theorem}

\begin{proof} 
	By boundedness of $S_{p}(\cdot)$, there exists $\tilde R>0$ such that, for all $u_{i},d_{i},\alpha_{i}, \beta_{i}, \gamma_{ik}, \delta_{ik},b_{ij}\in\bar U$, 
	$F_{ij}(\Zz)=-d_{i}z_{ij}+C_{ij}(\Zz)$, with $|C_{ij}(\Zz)|\leq \tilde R$. 
	For all $\Zz\in V$, it holds that
	$\frac{d}{dt}\frac{1}{2}\|\Zz\|^2=\sum_{i=1}^\Na\sum_{j=1}^\No z_{ij}\dot z_{ij}=\sum_{i=1}^\Na\sum_{j=1}^\No z_{ij}\Big(\!\!-d_{i} z_{ij}\!+\!C_{ij}(\Zz)\!+\! \!\frac{1}{\No}\sum_{l=1}^{\No}(d_{il} z_{il}-C_{il}(\Zz))\Big)=\Zz^{T} D \Zz + \sum_{i=1}^\Na\sum_{j=1}^\No z_{ij}\left(C_{ij}(\Zz) - \frac{1}{\No}\sum_{\substack{l=1}}^{\No}C_{il}(\Zz)\right)\leq \Zz^{T} D \Zz + \Na \No\tilde{R}\|\Zz\|$,
    where we have used $\sum_{j = 1}^{\No} z_{ij} = 0$ for all $i$. We compute
    $\Zz^{T} D \Zz =
        \sum_{i = 1}^{\Na} \sum_{j = 1}^{\No} \left( -d_{i}z_{ij}^{2}\right) \!+\! \frac{1}{\No}\! \sum_{i = 1}^{\Na} \sum_{l = 1}^{\No} d_{i}z_{il} \!\!\left(\sum_{j = 1}^{\No} z_{ij} \right)= \sum_{i = 1}^{\Na} \sum_{j = 1}^{\No} -d_{i}z_{ij}^{2} \leq -\min_{i}\{d_{i}\} \| \Zz \|^{2}.$
    Then, for all $\|\Zz\|\geq \frac{\Na\No\tilde R}{\min_{i} \{d_{i}\}}$, it follows that 
    $\frac{d}{dt}\frac{1}{2}\|\Zz\|^2\leq 
        -\|\Zz\|\left(\min_{i} \{d_{i}\}\|\Zz\|-\Na\No\tilde R \right) \leq 0$.
	 The result follows by \cite[Theorem 4.18]{Khalil2002}.
\end{proof}
These forward invariance and boundedness results lead to a natural connection of the opinion vector $\Zz_i \in \mathbf{1}_\No^{\perp}$ to a simplex vector $\mathbf{y}_{i} = (y_{i1}, \dots, y_{i \No})$, where $y_{ij} \geq 0$ for all $i,j$ and
$y{i1} + \dots + y_{i \No} = r, \;\;\; r > 0$,  
i.e. $\mathbf{y}_{i} \in \Delta$ where $\Delta$ is a $(\No-1)$-dimensional simplex. Define the simplex product space as $\mathcal{V} = \Delta \times \dots \times \Delta$. 

\begin{corollary}{\textbf{Mapping to the Simplex Product $\VV$.}}  \label{cor:mapping}
	Given a bounded set $\bar U\subset\R$, assume $u_{i},d_{i},\alpha_i, \gamma_{ik}, \beta_i, \delta_{ik},b_{ij}\in\bar U$, $i,k=1,\ldots,\Na$, $j,l=1,\ldots,\No$. 
	Then, the vector field of~\eqref{EQ:generic decision dynamics} 
	can be mapped from the forward invariant region $V\cap\{|z_{ij}|\leq R,\,i=1,\ldots,\Na,\,j=1,\ldots,\No\}$ to the product of simplex $\VV$ by the affine change of coordinates
	$L:V\cap\{|z_{ij}|\leq R,\,i=1,\ldots,\Na,\,j=1,\ldots,\No\} \to\VV\\
	\Zz \mapsto\frac{r}{\No R}\Zz+\frac{r}{\No}$, $r>0$.
\end{corollary}

The simplex product space $\mathcal{V}$ is often associated with models of opinion dynamics, e.g., in \cite{lorenz2006,sirbu_opinion_2013,Peng2015}. \newtext{ Under the mapping proposed in Corollary \ref{cor:mapping} or any other bijective mapping to the simplex product space (e.g. using the standard softmax function), the system state $\mathbf{y} = (\mathbf{y}_1, \dots, \mathbf{y}_{\Na}) \in \mathcal{V}$ can be interpreted as the \textit{absolute opinions} of agents that have equal voting capacity in the collective decision~\cite{AF-MG-AB-NEL:20}, or as probabilities of choosing a particular option.}  

\subsection{Proof of Theorem \ref{thm:clusters}}
	\label{app:clusters}
	Opinion dynamics (\ref{EQ:generic decision dynamics}) of agent $i \in \mathcal{I}_{p}$ are defined by
	\begin{align}\label{EQ:clustered_eqns}
	F_{ij}(\Zz)=-d_{p}z_{ij}  +  b_{pj} +\hspace{4.2cm}
	\end{align}
	$+ u_p ( S_{1} ( \bar \alpha_p z_{ij} + \tilde \alpha_p\sum_{k \in \mathcal{I}_{p} \setminus \{ i \} } z_{kj} 
	+ \sum_{\substack{s\neq p\\s=1}}^{N_{c}} \sum_{k \in \mathcal{I}_{s}} \tilde \gamma_{ps} z_{kj} ) 
 + \sum_{\substack{l\neq j\\l=1}}^\No S_{2} (  \bar \beta_{p} z_{il} + \tilde \beta_{p}\sum_{k \in \mathcal{I}_{p}  \setminus \{ i \}} z_{kl} 
	+ \sum_{\substack{s\neq p\\s=1}}^{N_{c}} \sum_{k \in \mathcal{I}_{s}} \tilde \delta_{ps} z_{kl} )) $.\\
Let $V_T(\Zz) = \sum_{p = 1}^{N_{c}} V_{p}(\Zz)$,  $V_{p}(\Zz) = \frac{1}{2}\sum_{i,k \in \mathcal{I}_{p}}
\sum_{j = 1}^{\No}(z_{ij} - z_{kj})^{2}$. Let 
$F_{ij}(\Zz)=-d_{i}z_{ij}+C_{ij}(\Zz)$. Then 
$\dot{V}_{p}(\Zz) = - \sum_{i \in \mathcal{I}_{p}}\sum_{k \in \mathcal{I}_{p}} d_p (\Zz_{i} - \Zz_{k})^{T}  (\Zz_{i} - \Zz_{k})+ \sum_{i \in \mathcal{I}_{p}}\sum_{k \in \mathcal{I}_{p}}\sum_{j = 1}^{\No}(z_{ij} - z_{kj})(C_{ij}(\Zz) - C_{kj}(\Zz))- \frac{1}{\No} \hspace{-0.5mm} \sum_{i \in \mathcal{I}_{p}}\sum_{k \in \mathcal{I}_{p}}\sum_{j = 1}^{\No} \sum_{l = 1}^{\No} (z_{ij} - z_{kj})(C_{il}(\Zz) - C_{kl}(\Zz))$.
The last term is zero because $\sum_{j = 1}^{\No}z_{ij} = 0$ on $V$. By the Mean Value Theorem, we can write $C_{ij}(\Zz) - C_{kj}(\Zz)$ in the second term as
$u_{p} \left( \kappa_{1} ( \bar \alpha_p - \tilde \alpha_p) - \kappa_{2}(\bar \beta_p - \tilde \beta_p ) \right)(z_{ij} - z_{kj})^{2}$,
where $\kappa_{1} \in K_{1}$ and $\kappa_{2} \in K_{2}$. Then we find that  $\dot{V}_{p}(\Zz) \leq \!\!\! \sup_{\kappa_{1} \in K_{1}, \kappa_{2} \in K_{2}}\! \! \Big\{- d_{p}  + u_{p} \kappa_{1} ( \bar \alpha_p - \tilde \alpha_p ) + u_{p} \kappa_{2} (\bar\beta_p - \tilde \beta_p ) \Big\} 2 V_{p}(\Zz) $.
When (\ref{eq:clustering_cond}) is satisfied, using LaSalle's invariance principle \cite[Theorem 4.4]{Khalil2002}  every trajectory of \eqref{EQ:generic decision dynamics} 
converges exponentially in time to the largest invariant set of $V_T(\mathbf{Z}) = 0$, which is $\mathcal{E}$. Let $\hat{z}_{pj} = z_{ij}$ for any $i \in \mathcal{I}_{p}$.  The dynamics (\ref{EQ:clustered_eqns}) on $\mathcal{E}$ reduce to (\ref{EQ:generic decision dynamics}) with $\Na = N_{c}$ and
weights (\ref{eq:cluster_weights}).

\newtext{
\begin{remark}\label{remark: symmetry clustering appe}
    This proof could alternatively be carried out using a group-theoretic approach outlined in \cite{sorrentino2016complete}. However this approach would only guarantee local stability of the clustered manifold, and the Lyapunov function approach presented here instead provides a global stability guarantee. 
\end{remark}}


\subsection{Proof of Theorem \ref{thm:Bifurcations}} \label{app:mainTheorem} \newtext{
$J = \left((-d + u\left(\alpha - \beta)\right)\mathcal{I}_{\Na}+u(\Gamma - \Delta)\right) \otimes P_0. $
Thus, eigenvalues of $J$ are of the form $\xi_i \lambda_o$, where $\xi_i$ is an eigenvalue of $(-d + u\left(\alpha - \beta)\right)\mathcal{I}_{\Na}+u(\Gamma - \Delta))$ and $\lambda_o$ is an eigenvalue of $P_0$ restricted to $V$. Because the only eigenvalue of $P_0$ restricted to $V$ is one, $\lambda_o=1$, whereas
$\xi_i=-d + u(\alpha - \beta)+u\lambda_i$,
where $\lambda_i$, $i=1,\ldots,\Na$ is an eigenvalue of $\Gamma - \Delta$. Thus, whenever $\alpha-\beta+\lambda>0$, all eigenvalues of $J$ are negative for $u<u^*$, zero is an eigenvalue of $J$ for $u=u^*$ (with multiplicity $(\No-1)N_{\lambda}$, where $N_{\lambda}$ is the multiplicity of $\lambda$), and there exist positive eigenvalues for $u>u^*$. The form of the eigenvectors of $J$ corresponding to its zero eigenvalue for $u=u^*$ follows since the eigenvectors of the Kronecker product of matrices is the Kronecker product of the eigenvectors. For simple $\lambda$, the statement follows from the Equivariant Branching Lemma~\cite[Section~1.4]{GolubitskySymmetryPerspective}.}


\subsection{Proof of Proposition \ref{PROP:realization symmetries}}
\label{app:transSym}

The proof of  (1) follows analogously to that of \cite[Theorem 2.5]{AF-MG-AB-NEL:20} with the additional coefficient $d_{i}$ on the linear terms. It is omitted due to space constraints.

To prove  (2), it is sufficient to show equivariance of the dynamics under the action of  generators of  
$S_{\No}\times D_{\Na}$. Element $\sigma \in S_{\No}$ acts on $V$ by permuting the elements 
of each agent's opinion  $\Zz_{i}$. Generators of $S_{\No}$ are $\No$ transpositions $\sigma_{j}$ where each $\sigma_{j}$ swaps adjacent 
elements $j$ and $j+1$ (or $\No$ and $1$ when $j = \No$). Let $\mathbf{F}_{i}(\Zz) = (F_{i1}(\Zz), \dots,  F_{i\No}(\Zz))$ and observe that $\sigma_{j} \mathbf{F}_{i}(\Zz) = (F_{i1}(\Zz), \dots, F_{i(j+1)}(\Zz), F_{ij}(\Zz), \dots,  F_{i\No}(\Zz))$. 
Computing $\mathbf{F}_{i}(\sigma_{j} \Zz)$, only $F_{ij}$ and $F_{i(j+1)}$ are changed, with
$F_{ij}(\sigma_{j} \Zz)\hspace{-0.5mm}=\hspace{-0.5mm}-d z_{i(j+1)}+
	u\Big(S_{1}\left(\alpha z_{i(j+1)}  +  \gamma  z_{(i-1)(j+1)} + \gamma z_{(i+1) (j+1)}\right)+ \hspace{-1mm}\sum_{\substack{l\neq (j+1)\\l=1}}^\No \hspace{-1mm} S_{2}\left( \beta z_{il}+ \delta z_{(i-1)(j+1)} + \delta z_{(i+1) (j+1)} \right) \Big)+\hat{b} $.
Thus, $\sigma_{j} \mathbf{F}_{i}(\Zz) = \mathbf{F}_{i}(\sigma_{j}\Zz)$ for all $j = 1...\No$, and for all $i = 1, \dots, \Na$, and the dynamics are equivariant under the action of $S_{\No}$. Element $\rho \in D_{\Na}$ acts on $V$ by permuting the order of the agent vectors $\Zz_{i}$ in the total system vector $\Zz = (\Zz_{1}, \dots, \Zz_{\Na})$. The generators of $D_{\Na}$ are the reflection element $\rho_{1}$ which reverses the order of elements in $\Zz$, and a rotation $\rho_{2}$ which cycles forward the vector by one element, mapping each element $i$ to $i+1$ (and $\Na$ to $1$). Let $\mathbf{F}(\Zz) = (\mathbf{F}_{1}(\Zz), \dots, \mathbf{F}_{\Na}(\Zz))$ and observe that $\rho_{1} \mathbf{F}(\Zz) = (\mathbf{F}_{\Na}(\Zz),\mathbf{F}_{\Na-1}(\Zz),  \dots,\mathbf{F}_{2}(\Zz),  \mathbf{F}_{1}(\Zz))$ and $\rho_{2} \mathbf{F}(\Zz) = (\mathbf{F}_{\Na}(\Zz),\mathbf{F}_{1}(\Zz), \mathbf{F}_{2}(\Zz), \dots,  \mathbf{F}_{\Na-1}(\Zz))$. For compactness we leave out the full expression for $F_{ij}(\rho_{p} \Zz)$. 

\bibliographystyle{IEEEtran}
\bibliography{references}

\end{document}